\documentclass[a4paper,11pt,parskip=half]{amsart}

\usepackage{amscd, amsfonts}
\usepackage[foot]{amsaddr}
\usepackage[T1]{fontenc}
\usepackage[utf8]{inputenc}
\usepackage{palatino}
\usepackage{lipsum}
\usepackage{amsmath, amssymb, mathtools, amsthm}
\usepackage{comment}

\usepackage[usenames,dvipsnames]{xcolor}
\usepackage{graphicx}
\usepackage{subfigure}
\usepackage{minitoc}
\usepackage{tikz}
\usepackage{enumerate}
\usepackage{filecontents}
\usepackage[margin=0.96 in]{geometry}
\usepackage{bbm}

\usepackage[numbers,sort&compress]{natbib}
\usepackage[colorlinks=true]{hyperref}
\hypersetup{urlcolor=blue, citecolor=red}
\usepackage{mathtools}

\usepackage[title]{appendix}
\usepackage{dsfont}

\newcommand{\rmspace}{\!\!\!}

\DeclarePairedDelimiter{\set}{\{}{\}}

\DeclarePairedDelimiter{\prt}{(}{)}

\newcommand{\curlyA}{\mathcal{A}}
\newcommand{\curlyB}{\mathcal{B}}
\newcommand{\curlyC}{\mathcal{C}}
\newcommand{\curlyD}{\mathcal{D}}

\newcommand{\curlyG}{\mathcal{G}}

\newcommand{\curlyI}{\mathcal{I}}

\newcommand{\curlyN}{\mathcal{N}}

\newcommand{\curlyZ}{\mathcal{Z}}

\newcommand{\OmegaT}{{\Omega_T}}

\newcommand \commentout[1] {}

\DeclareMathOperator*{\supp}{\operatorname{supp}}

\newcommand{\partialt}[1]{\dfrac{\partial#1}{\partial t}}

\DeclareMathAlphabet{\mathup}{OT1}{\familydefault}{m}{n}
\newcommand{\dx}[1]{\mathop{}\!\mathup{d} #1}

\theoremstyle{plain}
\newtheorem{theorem}{Theorem}[section]
\newtheorem{lemma}[theorem]{Lemma}

\newtheorem{corollary}[theorem]{Corollary}

\theoremstyle{remark}
\newtheorem{remark}[theorem]{\bf Remark}

\newcommand{\ie}{\emph{i.e.}}
\newcommand{\cf}{\emph{cf.}\;}

\newcommand{\sign}{\mathrm{sign}}

\newcommand{\ddt}{\frac{\dx{}}{\dx{t}}}

\newcommand{\grad}{\nabla}
\renewcommand{\div}{\nabla\cdot}
\newcommand{\Lap}{\Delta}

\newcommand{\n}{n_\gamma}
\newcommand{\p}{p_\gamma}
\newcommand{\f}{f_\gamma}

\newcommand{\R}{\mathbb{R}}

\begin{document}
\title[Incompressible Limit for Tissue Growth]{On The Incompressible Limit for a Tumour Growth Model Incorporating Convective Effects}

\author{Noemi David$^{1,2}$}
\author{Markus Schmidtchen$^{3}$}

\address{$^{1}$ Sorbonne Universit\'e, Inria, Universit\'e de Paris, Laboratoire Jacques-Louis Lions UMR 7598, Paris 75005, France (noemi.david@ljll.math.upmc.fr).}
\address{$^{2}$ Dipartimento di Matematica, Universit\`a di Bologna, Italy.}
\address{$^{3}$ Institute For Scientific Computing, Technische Universit\"at Dresden\\ 
Zellescher Weg 12-14, 01069 Dresden, Germany. (markus.schmidtchen@tu-dresden.de).}

\maketitle
\begin{abstract}
    In this work we study a tissue growth model with applications to tumour growth. The model is based on that of Perthame, Quir\'os, and V\'azquez proposed in 2014 but incorporates the advective effects caused, for instance, by the presence of nutrients, oxygen, or, possibly, as a result of self-propulsion. The main result of this work is the incompressible limit of this model which builds a bridge between the density-based model and a geometry free-boundary problem by passing to a singular limit in the pressure law. The limiting objects are then proven to be unique.
\end{abstract}

\vskip .4cm
\begin{flushleft}
    \noindent{\makebox[1in]\hrulefill}
\end{flushleft}
	Incompressible limit; Free boundary; Hele-Shaw problem.\\[-2.em]
\begin{flushright}
    \noindent{\makebox[1in]\hrulefill}
\end{flushright}
\vskip 1.5cm

\section{Introduction}
Modelling living tissue poses a whole range of challenges. On the one hand, it is important to identify the biomedical drivers that should be incorporated in the model, while, on the other hand there are certain modelling choices that need to be discussed. One of these choices that, in a way, separates the community is the type of model used to describe tissue growth. Roughly speaking we identify the following two types of models: those that describe the tissue as an evolving distribution in space and those that describe the tissue as an evolving domain in space. While the first type is mostly based on a partial differential equation  description, the latter is known as a free-boundary or evolving boundary model.\\

The goal of this paper is to build a bridge between the two types of models by passing to the so-called stiff limit in the population-based model to obtain a free-boundary description. The model we propose here describes the evolution of the tissue density, $\n=\n(x,t)$, and is given by
\begin{align}\label{eq:n}
    \partialt{\n} - \nabla\cdot \prt*{\n \nabla \p} -\nabla\cdot\prt*{\n \nabla \Phi} = \n G(\p). 
\end{align}
on $\R^d$ and for $t>0$. It is equipped with some non-negative initial data $\n(0,x)=\n^0(x) \in L_+^1(\R^d)$. Here $\p = \n^\gamma$ denotes the pressure, $G = G(\p)$ models the cell proliferation (resp. cell death), and  $\Phi = \Phi(x,t)$ denotes a chemical concentration. In order to pass to the incompressible limit  $\gamma \to \infty$ we need to study the equation satisfied by the pressure, \ie, in the equation
\begin{align}
    \label{eq:p}
    \partialt{\p} = \gamma \p (\Delta\p +\Delta\Phi +   G(\p) )+ \nabla \p \cdot\nabla(\p + \Phi). 
\end{align}
While it is intuitive to expect
\begin{align*}
     {p}_\infty (\Delta {p}_\infty +\Delta\Phi +   G({p}_\infty) ) = 0, \quad \text{as well as} \quad p_\infty(n_\infty-1) = 0,
\end{align*}
in the limit, there are technical subtleties, obtaining strong compactness of the pressure gradient to be precise, that need to be overcome. We are by no means the first to ask this question. As a matter of fact, there are already some promising results towards this rigorous limit. However, all of them are borderline and just not good enough to obtain the strong compactness of the pressure gradient. A blend of two techniques finally allows us to settle this open question. The rest of the introduction is dedicated to presenting a historical view on this type of model as well as variations thereof. We will also use this as an opportunity to introduce the tools necessary for the limit passage in a brief, explanatory way.

\subsection{Historical Notes -- The Origin of Incompressible Limits \& the Mesa Problem}
\label{sec:historical_notes}
The question of passing to the incompressible limit has a rich history and several variations of it have been studied in the literature. Historically, the problem has its early foundation in the work of B\'enilan and Crandall on the continuous dependence on $\varphi$ of solutions to the filtration equation $\partial_t n = \Delta \varphi(n)$ in 1981, \cf \cite{BC81}, not too long after the first wellposedness results for the filtration equation around 1960, \cf \cite{OKJ58, Sab61}. The continuous dependence of \cite{BC81} is established using nonlinear m-accretive semi-group theory, notably maximal monotone operators, enabling them to allow for cases of $\varphi$ being a monotone graph. As a matter of fact, it already covers the first result on incompressible limits by choosing $\varphi(z)=z^\gamma$ and assuming non-negative initial data bounded from above by unity.

Henceforth the problem has been attracting a lot of attention. It is worth noting that exponents  \\

In \cite{EHKO86} the authors show the formation of a plateau-like region, which they refer to as `mesa', of nearly constant density $n_\gamma$, for $\gamma \approx \infty$, using a formal asymptotic expansions and working with radial solutions. In \cite{CF87}, too, the authors consider the limit of the density of the porous equation but they can weaken the assumption on the initial data thus extending the results of \cite{BC81}. Moreover, they are able to show that the limit density, $n_\infty$, is independent of time and bounded $0\leq n_\infty \leq 1$. This `stationarity' result on the limit density, $n_\infty$, is obtained upon combining three tools. First, the uniform essential bounds on the compactly supported densities, $n_\gamma$ imply the weak-star convergence of a subsequence. Second, by the classical Aronson-B\'enilan estimate (see \cite{AB} for the original article as well as \cite{BPS2020} and references therein for a survey), it can be inferred that $\partial_t n_\infty\geq 0$, and therefore $n_\infty(x,t)\geq n_\infty(s,x)$ for almost every $x\in \R^d$, $s<t$, and all $\gamma>1$. Finally, the conservation of mass implies that, in fact, $n_\infty(x,t) = n_\infty(s,x)$, which shows that $n_\infty$ is independent of time, \cf \cite{BC81} for the full argument.\\

Later, in 2001, Gil and Quir\'os revisit the study of the incompressible limit of the solution of the porous medium equation defined in $[0,+\infty)\times\Omega$. In their paper they prove that the solution of the porous medium equation converges to that of the Hele-Shaw problem in the sense of Elliot and Janovsky, \ie, in the form of a variational formulation whenever the boundary data $g=g(x)$ is independent of time and the initial data is the indicator function of some bounded set $\Omega_0\subset\Omega$. In this case, the weak formulation  and the variational formulation coincide, \cf \cite[Corollary 4.5]{gil2001convergence}). In their study, \cf  \cite{gil2001convergence}, $\Omega$ is assumed to be a compact subset of $\R^d$ which is equipped with Dirichlet data on the pressure, $\p(t, x)=g(x,t)$ on $\partial \Omega$, for some $g(x,t)\geq 0$.  Let us point out that, given a set $\Omega$ large enough, the case $g\equiv 0$ coincides with the problem studied by Caffarelli and Friedman in \cite{CF87}, and, again, the limit is independent of time. Indeed, Gil and Quir\'os are able to recover the same result from a different perspective, focusing on the role of the pressure rather than the density itself. In the absence of Dirichlet boundary data, \ie, $g\equiv 0$, the limit solution solves a Hele-Shaw problem where the free boundary is actually motionless since the limit pressure vanishes almost everywhere. This can be easily seen by passing to the limit $\gamma \rightarrow \infty$ in the porous medium pressure equation, Eq. \eqref{eq:p}, where, of course, the growth term and the migration term are absent. In conjunction with the uniform essential bounds this immediately yields $\|\nabla p_\infty\|_{L^2(\Omega\times(0,T))}=0$. 

On the other hand, in the case non-vanishing $g\geq 0$ on $\partial\Omega$, the pressure is ``forced'' to be positive near to the boundary, and then, since the pressure gradient is no longer zero, the motion of the free boundary $\partial\{p_\infty>0\}$ is governed by Darcy's law 
$$
    V = -\partial_\nu p_\infty,
$$
where $\nu$ denotes the outward normal on the free boundary. In  \cite{GQ2003} the authors generalise there result towards a broader class of initial data give a description of the positivity set of the densities, $\n$, to that of the limit.

Let us also stress that the conservation of mass no longer holds since there is a source term on the boundary of $\Omega$. Therefore, the proof of the stationarity of $n_\infty$ using the Aronson-B\'enilan estimate fails. Similarly, the proof of $\|\nabla p_\infty\|_{L^2}=0$, no longer holds true due to the fact that the boundary terms arising from integration by parts no longer vanish.

It is also worthwhile noticing that $\p \approx \n \p$, for $\gamma \gg 1$, which leads to the relation 
$$
    p_\infty(1-n_\infty) = 0.
$$
Hence, we infer the inclusion $\{p_\infty>0\}\subset\{n_\infty=1\}$, but we also stress that the two sets need not coincide. In fact, in the case $g=0$, or equivalently the porous medium equation on $\R^d$ with compactly supported initial data, as mentioned above, the limiting pressure vanishes, $p_\infty = 0$, almost everywhere and the limit density is stationary, $n_\infty(x,t)=n^0(x),$ where $0\leq n^0(x) \leq 1$. This means that, even if there are saturation zones, $\set{n_\infty = 1}$, the pressure does not become positive. This situation changes drastically if the model includes a positive growth term of the form
\begin{align*}
     \partialt{\n} - \nabla\cdot \prt*{\n \nabla \p} = \n G(\p), 
\end{align*}
as was proposed in \cite{PQV}. In this case it can be shown that the two sets coincide, \ie, $\{p_\infty>0\} = \{n_\infty=1\}$, and, what is more, the problem is no longer stationary!

\subsection{Contemporary Advances -- Generalisations of the Model}
Emanating from the early works on the mesa problem for the porous medium equation, research began branching out in different directions. In this section we aim at giving a brief overview of different extensions of the porous medium equation, applications of the models obtained this way, as well as techniques used to study their respective incompressible limits analytically. 

The first generalisation concerns the inclusion of a pressure-dependent growth term proposed in the work of \cite{PQV}. Here the authors propose a tissue-growth model where cells move according to a population pressure generated by the total density of the form $p(n) = n^\gamma$. In conjunction with Darcy's law they recover the porous-medium type degenerate diffusion. In addition, they include a proliferation term, $n G(p)$, which models cells divisions with a pressure depending rate. Thus the proliferation rate, $G$, is assumed to be a decreasing function accounting for the fact that cells are less `willing' to divide in packed regimes, \cf Section \ref{sec:including_proliferation}.

The model was then extended by a nutrient distribution, $c(x,t)$, which is assumed to diffuse in the domain and released (resp. decayed) by general $L^2$-processes, \cf Section \ref{sec:including_nutrients}. Most recently, the inclusion of migratory processes, \ie, drift terms given by a velocity field,  $v(x,t)$, as a model extension received a lot of attention, \cf Section \ref{sec:including_drifts}. This is also where our  contribution to the current discourse enters, namely the first rigorous derivation of the complementarity relation, that is, an equation governing the pressure distribution inside of the moving boundary problem. Before we begin discussing our main result we shall also point out recent advances in the area of stiff-limits in the context of pressure laws that are different from Darcy's law, \cf Section \ref{sec:different_pressure_laws}. We conclude our short survey of the literature by mentioning some multi-phase results, where, instead of one equation, two interacting species are considered, \cf Section \ref{sec:multiphase_models}

\subsubsection{A Model including Proliferation}
\label{sec:including_proliferation}
\mbox{}\\
In \cite{PQV}, Perthame, Quir\'os, and V\'azquez propose the model
\begin{align}
    \label{eq:PQV_first_introduced}
    \partialt{\n} - \nabla\cdot \prt*{\n \nabla \p} = \n G(\p).
\end{align}
Their paper is seminal in that they were the first to perform the rigorous stiff pressure limit in the presence of growth terms. 
While strong compactness of the pressure is absolutely sufficient for the Hele-Shaw limit itself, obtaining the so-called complementarity relation which provides an equation for the pressure in the limit is much more involved. In fact, in order to obtain it strong compactness of the pressure gradient is indispensable. To this purpose, using the comparison principle, they show that the Laplacian of the pressure satisfies an Aronson-B\'enilan type estimate, namely $\Delta p + G(p) \gtrsim -C/\gamma t$.

In \cite{KP17} the authors study the same model through a viscosity solution approach. They are able to show that the density converges locally uniformly away from the free boundary $\partial\{p_\infty>0\}$. Moreover, they prove locally uniform convergence of the pressure (as long as the limit is continuous) and that $p_\infty$ is the viscosity solution of the Hele-Shaw problem 
\begin{equation}\label{eq: HS growth}
\begin{split}
    \left\{
    \begin{array}{rll}
    -\Delta p_\infty\!\!\!&= G(p_\infty),  &\text{ in } \{p_\infty>0\},\\[0.7em]
     V\!\!\!&= - \dfrac{|\nabla p_\infty|}{1-\min(1,n^E_\infty)}, &\text{ on } \partial\{p_\infty>0\},
     \end{array}
     \right.
     \end{split}
\end{equation}
where the normal velocity law was only formally presumed in \cite{PQV}, but not rigorously proven. Here, $n^E_\infty$ denotes the trace of the ``external'' limit density on the free boundary, namely the trace of $n_\infty$ from the set $\{n_\infty<1\}$. 

Let us stress the fact that, as the velocity law suggests, the density shows jump discontinuities at the free boundary. Moreover, the velocity blows up when the density reaches value 1, therefore, when new \textit{mesas} appear outside of $\{p_\infty>0\}$, the pressure becomes instantaneously positive in the new nucleated regions, hence exhibiting time discontinuities.

The free boundary problem, Eq. \eqref{eq: HS growth}, was further studied in \cite{MePeQu}, where the authors prove that the velocity law of the free boundary holds both in a weak (distributional) and in a measure theoretical sense. In the same paper, they also provide an $L^4$-bound of the pressure gradient that relies on the Aronson-B\'enilan estimate, which we extend to our model, Eq. \eqref{eq:n}, through a self-contained proof in Lemma~\ref{lemma:L4}, independently of any estimate on $\Lap \p$.

\subsubsection{A Model including Nutrients}
\label{sec:including_nutrients}
\mbox{}\\
In \cite{PQV}, the authors also study an extension of the model including the effect of a nutrient with concentration $c=c(x,t)$ in the growth term
    \begin{align*}
    \partialt{\n} - \nabla\cdot \prt*{\n \nabla \p} = \n G(\p, c_\gamma). 
\end{align*}
While they were able to prove the strong convergence of $\n$ and $c_\gamma$ as $\gamma\rightarrow\infty$, they leave open the question of how to recover the $L^2$-strong compactness of the pressure gradient needed to pass to the limit in the pressure equation and obtain the complementarity relation. 

This problem was addressed in \cite{DP}, where the authors combine a weak version of the Aronson-B\'enilan estimate in $L^3$ with a uniform bound of the pressure gradient in $L^4$ to infer strong compactness. In fact, the $L^\infty$-Aronson-B\'enilan estimate does not hold in the nutrient case, since $G(p,c)$ can be negative and then the comparison principle used in \cite{PQV} fails. Travelling waves solutions of the Hele-Shaw problem that arises in the stiff limit have been studied in \cite{PTV}. Besides, explicit solutions to the limit problem are presented in \cite{LTWZ2018} for initial data of the form of an indicator of a bounded set. Recently, interesting progress have been made in \cite{GKM2020} where the authors are able to establish the incompressible limit and the complementarity relation without relying on any Aronson-B\'enilan-type estimates. Instead, their approach is based on viscosity solutions and establishing the equivalence between the complementarity relation and an obstacle problem. 

\subsubsection{Models including local and non-local Drifts}
\label{sec:including_drifts}
\mbox{}\\
In 2010, Kim and Lei introduced the notion of viscosity solution for the porous medium equation with drift
\begin{equation*}
     \partialt{\n} =\Delta \n^\gamma +\nabla\cdot\prt*{\n \nabla \Phi},
\end{equation*}
 and they prove that it coincides with the weak solution in the distributional sense, \cf \cite{KL10}. 
Using the same viscosity approach, in \cite{AKY14} the authors study the link between the Hele-Shaw model with drift
\begin{align*} 
    \begin{split}
    \left\{
    \begin{array}{rll}
    -\Delta p\!\!\! &= \Delta \Phi,  &\text{ in } \{p>0\},\\[0.7em]
     V \!\!\!&= - \prt*{\nabla p+\nabla \Phi}\cdot\nu, &\text{ on } \partial\{p>0\}, 
     \end{array}
     \right.
    \end{split}
\end{align*}
and the congested crowd motion model 
$$
    \partial_t{n} + \nabla\cdot(n \nabla \Phi) = 0,
$$ 
if $n<1,$ with $n \leq 1$, where the latter constraint comes from the singular limit in the nonlinear diffusion term. To prove the equivalence of the two models, they study the asymptotics of the porous medium equation with drift as $\gamma \rightarrow \infty$. They show that the viscosity solution converges locally uniformly to a solution of the Hele-Shaw model. At the same time, using the metric setting of the 2-Wasserstein space, they infer the convergence to the aforementioned congested crowd motion model. To this purpose, they assume the potential $\Phi$ to be sub-harmonic, \ie, \ $\Delta \Phi >0$. While the convergence in the 2-Wasserstein distance holds for general initial data $0\leq n_0\leq 1$, the locally uniform limit holds only for \textit{patches}, namely $n^0=\mathds{1}_{\Omega_0}$, with $\Omega_0$ a compact set in $\R^d$. This result was extended in 2016,  by Craig, Kim, and Yao, \cf \cite{CKY18} to a model with non-local Newtonian potential, $\curlyN$,
\begin{equation*}
    \partialt \n = \Delta \n ^\gamma + \div(\n \nabla \curlyN \star \n).
\end{equation*}
The main novelty they introduce is that they are able to study the incompressible limit despite the lack of convexity. In fact, unlike the congested drift equation studied in \cite{AKY14}, the energy related to the aggregation equation through the 2-Wasserstein gradient flow structure is not semi-convex, \cf \cite{CKY18}.

The question of how to pass to the limit $\gamma \rightarrow \infty$ in the porous medium equation with a drift and a non-trivial source term has been addressed in \cite{KPW19}. The authors propose a model with a generic vector field $v:\R^d\times \R^+ \rightarrow \R^d$ as drift term, \ie,
\begin{equation*}
    \partialt{\n} - \Delta \n^\gamma + \div (\n \ v) = \n G,
\end{equation*}
with a growth rate $G=G(x,t)$. Through viscosity solutions methods, they prove that as $\gamma\rightarrow \infty$ the model converges to a free boundary model of Hele-Shaw type. Their work improves the results previously achieved in \cite{AKY14}, extending the class of initial data from patches to any continuous and compactly supported function bounded between zero and one. \\

\subsubsection{Different Pressure Laws and Relations}
\label{sec:different_pressure_laws}
\mbox{}\\
As foreshadowed above, in certain contexts Darcy's law may not be the appropriate relation that links the velocity field to the mechanical pressure. Depending on the modelling context and the model complexity, the pressure is incorporated in the fluid velocity through Stokes flow, Brinkman's law or Navier–Stokes' law, rather than Darcy's law. We briefly present recent works of incompressible limits for different pressure laws and relations.

\mbox{}\\
\textbf{Singular Pressure}\\
Parallel to the advances in the context of incompressible limits with power-law pressures it has been observed that another pressure law of the form
\begin{align}
    \label{eq:hecht_pressure}
    p_\epsilon(n) = \epsilon \frac{n}{1 - n},
\end{align}
can be used to model living tissue, \cf \cite{hecht2017incompressible}.  Using this singular pressure law already introduces an incompressibility condition in the sense that the pressure blows up when the cell density reaches the saturated regime, $n=1$. Thus, singular pressure laws of this kind are encountered in scenarios when non-overlap conditions are enforced already at a population-level, \cf \cite{DHN2011, PZ2015} in the context of congestive collective crowd motion, \cite{BDDR2008,BDLMRR2008} in the context of traffic flow modelling.
In \cite{hecht2017incompressible} the authors are able to show that the pressure in Eq. \eqref{eq:hecht_pressure} is suitable to pass to the incompressible limit using a generalisation of the Aronson-B\'enilan argument by Crandall and Pierre, \cf \cite{crandall1982regularizing}. 

\mbox{}\\
\textbf{Brinkman Law Pressure}\\
Unlike Darcy's law using the Brinkman law,
\begin{align*}
    -\nu \Delta W + W = p(n),
\end{align*}
accounts for visco-elastic effects, \cite{BD2009}. Based on this observation, in \cite{PV2015} the authors propose a modification of the above model, Eq. \eqref{eq:PQV_first_introduced}, incorporating the Brinkman law, \ie,
\begin{align*}
    \partialt{\n} - \nabla\cdot \prt*{\n \nabla W_\gamma} = \n G(\p).
\end{align*}
Different from the Darcy law setting the authors are forced to use a different set of techniques since the problem is no longer degenerate parabolic but, instead, of transport nature. While, at first glance, the Brinkman law has a regularising effect on the velocity field it makes obtaining compactness of the pressure a hard endeavour. Using a kinetic reformulation and controlling oscillations in the pressure finally yields the required compactness to pass to the incompressible limit and obtain a visco-elastic version of the complementarity relation, \cf \cite[Theorem 1.1]{PV2015}. For pressure laws of the form $p_\epsilon(n) = \epsilon \mathds{1}_{n\geq 1} \log(n)$, quite recently, explicit travelling wave profiles we obtained by \cite{LTWZ2019}. Moreover, the authors provide an apt numerical scheme to track the moving front accurately.

\mbox{}\\
\textbf{Stokes Flow}\\
It is important to stress that both Darcy's law and Brinkman's law are, at least, formally related to the Navier-Stokes law which can therefore be seen as the most general relation between the fluid velocity and the mechanical pressure. In \cite{VaucheletZatorska} the authors prove the incompressible limit for a proliferating species whose velocity is linked to the pressure through the Navier-Stokes law thus generalising the case without birth and death processes of \cite{LM1999}. The authors use the fact that the growth rate is linear in the pressure such that weak compactness of the pressure suffices in order to pass to the limit, so long as the density itself is strongly compact. While the weak compactness of the pressure follows from a renormalisation argument the strong compactness of the density is based on a compactness-propagation argument introduced (and later refined) in \cite{BJ2013, BJ2017, BJ2018}.

\mbox{}\\
\textbf{Active Motion}\\
In \cite{activemotion} the authors extend the model of \cite{PQV} by an additional active motion term in form of a linear diffusion term. They are able to rigorously perform the incompressible limit, in fact they obtain the same complementarity relation as in the absence of active motion, for certain initial data not relying on the Aronson-B\'enilan for certain initial data. Nonetheless, the restriction on the initial data can be dropped by employing the argument of Crandall and Pierre, in \cite{CrPi1982}. In \cite{TVCVDP2014} the authors propose a very similar model based on Brinkman's law (unlike \cite{activemotion}) including a linear diffusion term. They observe that travelling waves exist and analyse their profile.

\mbox{}\\
\textbf{Fractional Diffusion}\\
In 2015, J.-L. V\'azquez opened another both fascinating and challenging research direction by addressing the mesa problem in the fractional pressure case, \cf \cite{Vaz2015}. More precisely, he studies the incompressible limit, $\gamma \to \infty$, in the fractional porous medium equation,
\begin{align*}
    \partialt{\n} + \prt*{-\Delta}^{-s} \prt*{\n}^\gamma = 0,
\end{align*}
for $s\in(0,1)$. Unlike the case of classical porous medium type diffusion, the limiting profile exhibits tails and does not remain compactly supported. The analysis is of orders of magnitude harder since the classical theory discussed in Section \ref{sec:historical_notes} relies on comparison principles and the fact that it is known what happens to the Barenblatt profiles in the incompressible limit. In the fractional setting the explicit source solutions are not known explicitly. None the less, they are the starting point of the analysis of \cite{Vaz2015}. Many questions remain open, in particular the inclusion of other processes such as reactions and drifts.

\mbox{}\\
\subsubsection{Multi-Species System}
\label{sec:multiphase_models}
Recently, there has growing interest in multi-phase extensions of the above model. Instead of merely modelling the evolution of a single species, say, cancer tissue, other phases such as interstitial fluid, healthy tissue, dead tissue, \ldots, are incorporated into the model. The extension to multiple interacting species not only leads to interesting behaviours such as phase separation but also raises novel mathematical challenges such as the loss of regularity at so-called internal layers, \ie, regions where two or more phases get in contact. Recently, \cite{BPPS} have established the rigorous incompressible limit for a two-species model consisting of normal and abnormal tissue, respectively for a Darcy law type pressure. Unlike in the single-species case, the pressure is now generated by the joint population in form of a power law. However, the lack of regularity is such that only a one dimensional result could be obtained and the general case was successfully addressed only recently, \cf \cite{LX2021}. In a similar fashion, a one-dimensional result could be obtained, see \cite{DHV},  when the pressure is given by the singular law, Eq. \eqref{eq:hecht_pressure} using the generalisation of the Aronson-B\'enilan estimate introduced in \cite{CrPi1982}.  

A more complete picture is available if the cells do not avoid overcrowding due to Darcy's law but if they move according to Brinkman's law. Coupling the cell's `velocity' to the pressure accounts for visco-elastic effects, \cf  \cite{DeSc, DEBIEC2020}. A coupling through the more general Stoke's flow remains a challenging open problem. Recently, \cite{DLZ2020} proposed a two-cell-type model coupled with nutrients to study the effect of autophagy on tumour growth. In their work they, too, consider an incompressible limit, however the results remains formal due to difficulties similar to that of the system without nutrients treated by \cite{DHV, BPPS}.

\subsection{Our Contribution}
As set out in the introduction, there have been several promising steps towards establishing the incompressible limit and the complementarity relation for reaction-diffusion models incorporating convective effects. 
As a matter of fact, just like the authors of \cite{KPW19}, we address the problem of passing to the incompressible limit in a porous medium equation with both a drift and a source term. While their approach is based on a viscosity solution approach, we use a weak (distributional) interpretation. By employing a blend of recently developed tools, \ie, an $L^p$-version of the celebrated Aronson-B\'enilan estimate, \cf \cite{AB}, along with the optimal $L^4$-regularity of the pressure gradient observed in \cite{DP}, we can obtain strong compactness of the pressure gradient and proceed to passing to the incompressible limit and obtain the complementarity relation in the same vein as \cite{BPPS}. To summarise:
\begin{itemize}
    \setlength\itemsep{1em}
    \item We obtain an $L^3$-space-time estimate on the negative part of the Laplacian of the pressure which ultimately helps us obtain strong compactness of the pressure gradient. We note that an $L^\infty$-version has been obtain recently in \cite[Theorem 3.1]{KZ2020}. However, the lower bound on the Laplacian of the pressure that they infer, $\Lap p \geq - C/t - C$, does not go to zero as $\gamma \rightarrow \infty$, as in the classical Aronson-B\'enilan estimate. Nonetheless, this result in conjunction with our uniform $L^4$-estimate on the pressure gradient would already be sufficient to obtain the complementarity relation rigorously, following \cite{DP, BPPS, MePeQu}. 
    \item Here, we choose a different route by only striving for the much weaker $L^3$-estimate on the negative part of the Laplacian of the pressure. This, in turn, allows us to drastically relax the  $C^{3,1}_{x,t}$-regularity of the velocity field, $\nabla \Phi$, required by \cite{KZ2020}. In fact, our assumptions on the drift, \cf Eq. \eqref{eq:assptn_phi_BV} and Eq. \eqref{eq:assptn_phi_AB}, in a way boil down to controlling certain third derivatives in $L_{\mathrm{loc}}^{12/5}(Q_T)$.
    \item Finally, to the best of our knowledge, we are the first to prove the uniqueness of the solution, $(n_\infty,p_\infty)$, to the limit problem 
\begin{equation*} 
        \partialt{n_\infty}  = \Delta p_\infty + n_\infty G(p_\infty) + \nabla \cdot (n_\infty \nabla \Phi).
\end{equation*}
This result is only possible since we work with weak solutions in the classical sense which ultimately allows us to apply a variation of Hilbert's duality method. The only related result in this direction in the literature is given by \cite{AKY14} where the uniqueness of so-called patch solutions is shown in the drift-diffusion model with $\Delta \Phi >0$ in the absence of growth dynamics.
\end{itemize}
Moreover, our approach provides an answer to several open problems proposed in \cite{KPW19}:
\begin{itemize}
    \setlength\itemsep{1em}
    \item The first question the authors raise concerns the monotonicity assumption on $G(p) + \Lap \Phi > 0$, which in our case is not necessary. An improvement in this direction has also been obtained very recently, \cite{GKM2020}. We stress that in the growth rate in \cite{KPW19} does not depend on the pressure but on space and time, only.
    \item The next question concerns the class of initial data. In \cite{KPW19}, the authors write ``A more interesting question arises with the initial data that is larger than $1$ at some points. In such cases there is a jump in the solution at $t=0$ in the limit `$\gamma \to \infty$' which adds another challenge in the analysis.''\footnote{This quote is directly taken from \cite{KPW19} where we only adapted the notation to that of our paper.} This effect has already been observed at the early stages of this singular limit problem. The parts of the density that are larger than $1$ are known to ``collaps'' immediately and a mesa-structure is obtained instantaneously, for instance, \cf \cite{CF87}. Following our approach, we can allow for the larger class of  non-negative $L^1(\R^d) \cap L^\infty(\R^d)$ functions with compact support as initial data.\footnote{While $L^\infty$-data with compact support immediately implies integrability, we trust that the assumption on the support may be removed by a localising argument in the spirit of \cite{DP, GPS}.} 
    \item Finally, in \cite{KPW19}, the authors postulate $BV$-regularity of the limiting density, also suggested by \cite{DMSV2016} based on the ``five-gradients-estimate'' using tools from optimal transportation. Even though our arguments do not borrow techniques from optimal transport but, instead, rely on Sobolev compactness theory, we are able to improve the regularity result in that we obtain the $BV$-regularity of the limit density for any initial data. What is more, we additionally have an $L^{4}$-regularity of the limit pressure gradient, which, to the best of our knowledge, is novel. 
\end{itemize}

\subsection{Problem Setting and Main Results}
Before we present the main results of our paper let us introduce some notation used throughout this work. Henceforth, we call $Q_T:=\R^d \times (0,T)$ the truncated space-time cylinder and drop the subscript $T$ to denote the entire cylinder, \ie,  $Q:= \R^d \times (0,\infty)$. Besides, for the sake of readability, we shall employ the short-hand notation
$$
    \n := \n(t) := \n(x,t),
$$
and, similarly,  
$$
    \p = \p(t):=\p(x,t).
$$
Moreover, throughout, $C>0$ denotes a generic positive constant independent of $\gamma$ that may change from line to line.\\

In order to be able to establish our result we impose the following set of assumptions which, for clarity, are split into assumptions on the initial data, the growth terms, and the advective term, respectively. 

We assume that for every $\gamma>1$ the initial data are non-negative, integrable, and uniformly essentially bounded, \ie, 
\begin{align}
    \label{eq:assptn_init1}
    \tag{A1-$\n^0$}
    \n^0\in BV(\R^d)\cap L^\infty(\R^d), \quad  0\leq \n^0\leq n_M, \quad  \mathrm{and} \quad  0\leq \p^0\leq p_M,
\end{align}
for some constants $n_M, p_M>0$. Here $BV$ denotes the space of functions with bounded variation. Moreover, we assume the initial population is contained in a compact set, \ie, there exists a bounded set $K\subset \R^d$ such that
\begin{align}
    \label{eq:assptn_init2}
    \tag{A2-$\n^0$}
    \supp(\n^0) \subset K.
\end{align}
Let us notice that, thanks to the finite speed of propagation property of porous medium type equations, assumption \eqref{eq:assptn_init2} implies that, for any $T>0$, there exists a bounded domain $\Omega\subset \R^d$ such that the supports of $\n(\cdot, t), \p(\cdot, t)$ are contained in $\Omega$ for any $t\in[0,T]$, uniformly in $\gamma$, as proven in the next section, \cf Lemma~\ref{lemma:aprioriestimates}.   

In addition, we suppose that there exists a positive constant $C$ independent of $\gamma$ such that 
\begin{align}
    \label{eq:assptn_init3}
    \tag{A3-$\n^0$}
   \|\Delta (\n^0)^{\gamma+1}\|_{L^1(\R^d)} 
   + \|\nabla \p^0\|_{L^2(\R^d)} + \||\Delta p^0|_-\|_{L^2(\R^d)}\leq C. 
\end{align}
Note, that strictly speaking, the $L^2$-bound on the pressure gradient is not required as it is a consequence of the $L^2$-control on the Laplacian of the pressure. Besides we make the biological assumption
\begin{align}
    \label{eq:assptn_growth}
    \tag{A-$G$}
    G'(p) < -\alpha, \quad \mathrm{and} \quad G(p_M) = 0,
\end{align}
for some $\alpha>0$ and all $p \in [0, p_M]$, \ie, to include the tendency of tissue to grow slower as the pressure increases and starts to die when the pressure exceeds the homeostatic pressure, $p_M$.
Finally, we have to make the following regularity assumptions on the chemical distribution  
\begin{align}
    \label{eq:assptn_phi_BV}
    \tag{A1-$\Phi$}
    \left\{
    \begin{array}{rl}
        \nabla \prt*{\partial_t \Phi} \in& \rmspace L^1((0,T); L_{\mathrm{loc}}^\infty(\R^d)),\\[0.8em]
        \Delta \prt*{\partial_t \Phi} \in& \rmspace L_{\mathrm{loc}}^1(Q_T),\\[0.8em]
          D^2\Phi \in& \rmspace  L_{\mathrm{loc}}^\infty(Q_T),\\[0.8em]
         \grad \Phi \in&  \rmspace  L_{\mathrm{loc}}^2(Q_T)\cap L_{\mathrm{loc}}^\infty(Q_T),     
    \end{array}
    \right.
\end{align}
and
\begin{align}
    \label{eq:assptn_phi_AB}
    \tag{A2-$\Phi$}
    \begin{array}{rl}
         \grad (\Lap \Phi) \in & \rmspace  L_{\mathrm{loc}}^{12/5}(Q_T).
    \end{array}
\end{align}
Note, that the additional assumption, (A2-$\Phi$), is required solely for technical reasons to establish the control of the Laplacian of the pressure.

Under these hypotheses we are now able to state the two main theorems of this work. The first concerns the complementarity relation.
\begin{theorem}[Complementarity relation]  
We may pass to the limit in Eq. \eqref{eq:p} as $\gamma \rightarrow \infty$ and establish the so-called \textit{complementarity relation}
\begin{equation} \label{eq:cr}
    p_\infty (\Lap p_\infty + \Lap \Phi + G(p_\infty))=0,
\end{equation}
in the distributional sense. Moreover, $0\leq n_\infty\leq 1$ and $p_\infty\geq 0$ satisfy the equation
\begin{subequations}
\begin{equation} \label{eq:limitproblem}
        \partialt{n_\infty} = \Delta p_\infty + n_\infty G(p_\infty) + \nabla \cdot (n_\infty \nabla \Phi),
\end{equation}
in $\curlyD'(Q_T)$, as well as
\begin{equation}
\label{eq:limitproblem2}
    p_\infty(1-n_\infty) = 0,
\end{equation}
almost everywhere.
\end{subequations}
\end{theorem}

The complementarity relation, Eq.  \eqref{eq:cr}, is a crucial link that allows us to bridge the gap between the compressible model, Eq. \eqref{eq:n}, and the geometrical free boundary problem of Hele-Shaw type. Let us define the set
$$
    \Omega(t):=\{x \ |\  p_\infty(x,t)>0\}.$$
Then, the pressure satisfies
\begin{equation*} 
   \left\{
  \begin{aligned}
    -\Delta p_\infty&= \Delta \Phi + G(p_\infty), \quad &&\text{ in } \Omega(t),\\[0.4em]
    p_\infty&= 0, \quad &&\text{ on } \partial\Omega(t),
    \end{aligned}
    \right.
\end{equation*}
which coincides with the classical Hele-Shaw problem whenever $\Phi$ and $G$ are identically equal to zero.

\begin{theorem}[Uniqueness of the limit solution]
There exists at most one distributional solution such that for all $T>0$ the couple $(n_\infty, p_\infty) \in L^\infty(Q_T)\times L^2(0,T; H^1(\Omega))$ is a solution to system \eqref{eq:limitproblem}.
\end{theorem}

The rest of the paper is organised as follows. In Section \ref{sec:apriori} we present straigh-forward a priori estimates necessary to derive more refined bounds on the pressure. The latter are proven in Section \ref{sec:strongerest}. This includes both the $L^3$-version of the Aronson-B\'enilan estimate as well as an $L^4$-space-time estimate on the pressure gradient.
Building on the estimates derived in the previous sections, Section \ref{sec:incomp_limit} is dedicated to the rigorous limit process in the pressure equation and to obtaining the complementarity relation. In the subsequent section, Section \ref{sec:uniq_limit}, we then proceed to proving the uniqueness of solutions to the complementarity relation.

\section{A Priori Estimates}
\label{sec:apriori}
We state some a priori estimates on the main quantities and their derivatives, that we need to obtain the main result of the paper. 
\begin{lemma}[A priori estimates]\label{lemma:aprioriestimates}
For any $T>0$, there exists a bounded domain $\Omega\subset \R^d$ such that the supports of $\n(\cdot, t), \p(\cdot, t)$ are contained in $\Omega$ for any $t\in[0,T]$, uniformly in $\gamma$. Moreover, the following estimates hold uniformly in $\gamma$:
\begin{enumerate}[(i)]
    \setlength\itemsep{0.7em}
    \item $\n,\p \in L^{\infty}(0,T;L^\infty(\Omega))$,\\[-0.9em]
    \item $\partial_i \n, \partial_t \n \in L^{\infty}(0,T;L^1(\Omega))$,  for $i=1,\dots,d$,\\[-0.9em]
    \item $\partial_i \p, \partial_t \p \in L^{1}((0,T)\times\Omega)$, for $i=1,\dots,d$,\\[-0.9em]
    \item $\grad \p \in L^2(0,T;L^2(\Omega))$.
\end{enumerate}
\end{lemma}
\begin{proof}

Thanks to the comparison principle, from Eq. \eqref{eq:n} we immediately find $\n \geq 0$ and, as a consequence, $\p\geq 0$. In order to establish uniform essential bounds, we construct a super solution. To this end we define 
\begin{equation*}
    \Pi(x,t):= C \left| R(t) - \frac{|x|^2}{2}\right|_+
\end{equation*}
where $C$ is a positive constant that satisfies
\begin{equation}\label{eq:C} 
    C \geq \frac2d (G(0)+\|\Delta \Phi\|_{\infty}),
\end{equation}
and we take $R(t)$ such that
\begin{equation}\label{eq: R'}
        R'(t)\geq (2 C + 1) R(t) + \frac{\|\nabla \Phi\|_\infty}{2}.
\end{equation}
From Eq. \eqref{eq:p} and the assumption on the growth term \eqref{eq:assptn_growth}, we know that $\p$ satisfies
\begin{equation*}
    \partialt \p - |\nabla \p|^2 - \nabla \p \cdot \nabla \Phi - \gamma \p (\Delta \p + G(0) + \|\Delta \Phi\|_{\infty}) \leq 0.
\end{equation*}
Let us show that $\Pi(x,t)$ is a super-solution to this differential inequality. We have
\begin{align*}
    \partialt \Pi &= C R'(t) \mathds{1}_{\left\{R(t)\geq \frac{|x|^2}{2}\right\}},
\end{align*}
and
\begin{align*}
    \nabla \Pi &= - C x \mathds{1}_{\left\{R(t)\geq \frac{|x|^2}{2}\right\}},
\end{align*}
as well as
\begin{align*}
    \Delta \Pi &= - C d \mathds{1}_{\left\{R(t)\geq \frac{|x|^2}{2}\right\}} - C |x| \delta_{\left\{R(t)= \frac{|x|^2}{2}\right\}}.
\end{align*}
Using Eq. \eqref{eq:C} in conjunction with Eq. \eqref{eq: R'} we get
\begin{align}
\label{eq:Phi}
\begin{split}
    \partialt\Pi - |\nabla\Pi|^2 - &\nabla\Pi\cdot\nabla\Phi - \gamma \Pi (\Delta \Pi + G(0) + \|\Delta \Phi\|_\infty)\\[0.7em]
    \geq &C R'(t) \mathds{1}_{\left\{R(t)\geq \frac{|x|^2}{2}\right\}} - C^2 |x|^2 \mathds{1}_{\left\{R(t)\geq \frac{|x|^2}{2}\right\}} + C x \cdot\nabla\Phi \mathds{1}_{\left\{R(t)\geq \frac{|x|^2}{2}\right\}} + \gamma C \Pi \frac d 2\\[0.7em]
    \geq & \left(R'(t) - 2C R(t) -\frac{|x|^2}{2} - \frac{\|\nabla \Phi\|_\infty}{2}\right)\mathds{1}_{\left\{R(t)\geq \frac{|x|^2}{2}\right\}}\\[0.7em]
    \geq &0.
    \end{split}
\end{align}
Taking $R(0)$ such that $K\subset B_{\sqrt{2R(0)}}$ and $C$ large enough, by the assumption on the initial data \eqref{eq:assptn_init2} we have $\p^0 \leq \Pi(0)$. Then, this implies that $\p(t) \leq \Pi(t)$ for all positive times by comparison. Let us show the argument for the sake of completeness. 

Setting $N(\Pi)=\Pi^{1 /\gamma}$, and multiplying Eq. \eqref{eq:Phi} by $N'(\Pi)$ we obtain
\begin{equation*}
    \frac{\partial N}{\partial t} - N'(\Pi) |\nabla \Pi|^2 - N'(\Pi) \nabla \Pi \cdot \nabla \Phi  - \gamma N'(\Pi) \Pi\Delta \Pi \geq \gamma N'(\Pi) \Pi (G(0)+\|\Delta \Phi\|_\infty),
\end{equation*}
whence
\begin{equation*}
    \frac{\partial N}{\partial t} - \nabla\cdot (N \nabla\Pi)- \nabla N \cdot \nabla \Phi \geq N (G(0)+\|\Delta \Phi\|_\infty).
\end{equation*}
Since, by Eq. \eqref{eq:n}, we know that $\n$ is a sub-solution to the same equation,   we have $\n(t) \leq N(t)$ for all $t>0$, by the comparison principle. Therefore, we conclude that $\p(t) \leq \Pi(t)$ for all positive times. We take $\Omega \subset\R^d$ a bounded domain such that $B_{\sqrt{2R(T)}}\subset\Omega$ and then, by the definition of $\Pi$, we infer that
\begin{equation*}
    \supp(\p(t)) \subset \Omega,
\end{equation*}
for all $t \in [0,T]$ and any $\gamma>1$.
As consequence, both $\n$ and $\p$ are uniformly bounded in $L^\infty(\Omega_T),$ where $\Omega_T:=\Omega\times(0,T)$.

Now we prove the $BV$-estimates on the density. Differentiating Eq. \eqref{eq:n} with respect to the $i$-th component of the space variable, $x_i$, and multiplying by $\sign(\partial_{x_i} \n)$ we get
\begin{align*}
    \ddt\!\int_{\Omega}\left|\frac{\partial \n}{\partial{x_i}}\right|\dx{x} \leq &\int_{\Omega} \!\!\gamma \Delta \left(\n^\gamma  \left|\frac{\partial \n}{\partial x_i}\right|\right) \dx{x} \!+\! \int_{\Omega} \!\!\div \!\prt*{\!\n \grad \prt*{\!\frac{\partial\Phi}{\partial{x_i}}\!}\!\!} \ \sign\prt*{\frac{\partial \n}{\partial{x_i}}}\dx{x} + G(0)\int_{\Omega} \left|\frac{\partial \n}{\partial{x_i}}\right|\dx{x}\\[1em]
    \leq &\sum_{j=1}^d \int_{\Omega} \left|\frac{\partial \n}{\partial{x_j}}\right| \left|\frac{\partial^2\Phi}{\partial{x_i \partial x_j}}\right|\dx{x} + \sum_{j=1}^d \int_{\Omega} \n \left|\frac{\partial^3\Phi}{\partial{x_i} \partial {x^2_j}}\right|\dx{x} + G(0)\int_{\Omega} \left|\frac{\partial \n}{\partial{x_i}}\right|\dx{x},
\end{align*}
for $i=1,\dots,d$.
We sum the inequalities over all $i=1,\dots,d$, and  obtain
\begin{align*}
    \ddt \sum_{i=1}^d \int_{\Omega}\left|\frac{\partial \n}{\partial{x_i}}\right|\dx{x} \leq C \sum_{i=1}^d \int_{\Omega} \left|\frac{\partial \n}{\partial{x_i}}\right|\dx{x} + C,
\end{align*}
where the constants depend on the $L^\infty$-norm of $G$ and the assumptions on the potential $\Phi$, \cf Eqs. (\ref{eq:assptn_growth}, \ref{eq:assptn_phi_BV}). Using Gronwall's lemma we conclude
\begin{equation*} 
    \sum_{i=1}^d \int_{\Omega}\left|\frac{\partial \n}{\partial{x_i}}\right|\dx{x} \leq Ce^{C t}  \sum_{i=1}^d \int_{\Omega}\left|\frac{\partial \n^0}{\partial{x_i}}\right|\dx{x} \leq C(T),
\end{equation*}
where, in the last inequality, we have used  the uniform $BV$-bounds on the initial data, \cf assumption \eqref{eq:assptn_init1}. 

Following the same line of reasoning for the time derivatives we obtain
\begin{align}
\label{dtn}
\begin{split}
    \partialt{} \left|\partialt \n\right| \leq
    &\gamma\Lap\prt*{\p \left|\partialt \n\right|} + \div \prt*{\left|\partialt \n\right| \grad \Phi} + \sign\prt*{\partialt \n}\nabla\cdot\prt*{\n \nabla \prt*{\partialt \Phi}}\\[0.7em]
    &+ \left|\partialt \n\right| G(\p) + \n G'(\p)\left|\partialt \p\right|,
\end{split}
\end{align}
due to the fact that $\sign(\partial_t \p)= \sign(\partial_t \n)$. An integration in space yields
\begin{equation*}
    \ddt\int_{\Omega}\left|\partialt \n\right|\dx{x} \leq G(0)\int_{\Omega} \left|\partialt \n\right|\dx{x}+\underbrace{\int_{\Omega}\left|\nabla\cdot\prt*{\n \nabla \prt*{\partialt \Phi}}\right|\dx{x}}_{\curlyI},
\end{equation*}
where we used that $G'<-\alpha$, \cf Eq. \eqref{eq:assptn_growth}. We can estimate the term $\curlyI$ as follows
\begin{align*}
    \curlyI =& \int_{\Omega}\left|\nabla\n \cdot \nabla \prt*{\partialt \Phi} + n \Delta\prt*{\partialt\Phi}\right|\dx{x}\\[0.7em]
    \leq& \int_{\Omega} \left|\nabla \n \cdot \nabla \prt*{\partialt\Phi}\right| \dx{x} + \int_{\Omega} \left| n \Delta\prt*{\partialt \Phi}\right|\dx{x}\\[0.7em]
    \leq& \left\|\nabla\prt*{\partialt \Phi}(\cdot, t)\right\|_{L^\infty(\Omega)} \|\nabla \n\|_{L^\infty(0,T;L^1(\Omega))} + n_H \left\|\Delta\prt*{\partialt \Phi}(\cdot, t)\right\|_{L^1(\Omega)} \\[0.7em]
    \leq&  C \left\|\nabla\prt*{\partialt \Phi}(\cdot, t)\right\|_{L^\infty(\Omega)} +C \left\|\Delta\prt*{\partialt \Phi}(\cdot, t)\right\|_{L^1(\Omega)},
\end{align*}
where we have used the $BV$-space regularity of $\n$ from before. Hence, we obtain
\begin{equation*}
    \ddt\int_{\Omega}\left|\partialt \n\right|\dx{x} \leq G(0)\int_{\Omega} \left|\partialt \n\right|\dx{x} +C \left\|\nabla\prt*{\partialt \Phi}(\cdot, t)\right\|_{L^\infty(\Omega)} + C\left\|\Delta\prt*{\partialt \Phi}(\cdot, t)\right\|_{L^1(\Omega)}.
\end{equation*}
By assumption \eqref{eq:assptn_phi_BV} we know that $\left\|\nabla\prt*{\partial_t \Phi}(\cdot, t)\right\|_{L^\infty(\Omega)}$ and $\left\|\Delta\prt*{\partial_t \Phi}(\cdot, t)\right\|_{L^1(\Omega)}$ are $L^1$-integrable in time. Using Gronwall's lemma, we conclude
\begin{align}
\label{eq:L1n}
\begin{split}
    \left\|\partialt \n(t)\right\|_{L^1(\Omega)} 
    &\leq e^{G(0) t} \left\|\prt*{\partialt \n}_0\right\|_{L^1(\Omega)} \\[0.7em]
    &\qquad + \int_0^t C\left(\left\|\nabla\prt*{\partialt \Phi}(\cdot, t)\right\|_{L^\infty(\Omega)} + \left\|\Delta\prt*{\partialt \Phi}(s, \cdot)\right\|_{L^1(\Omega)}\right) e^{G(0)(t-s)}\dx{s}\\[0.7em]
    & \leq C(T),
\end{split}
\end{align}
for a.e. $t\in (0,T)$, \ie, $\partial_t \n \in L^{\infty}(0,T; L^1(\Omega))$. Let us stress that assumptions \eqref{eq:assptn_init1} and \eqref{eq:assptn_init3} imply the initial bound $\left\|\prt*{\partial_t \n}_0\right\|_{L^1(\Omega)} \leq C$.

Before establishing the $BV$-bounds on the pressure, let us notice that integrating Eq. \eqref{dtn} in space and time, we have
\begin{equation*}
    \left\|\partialt \n(\cdot, t)\right\|_{L^1(\Omega)} + \min_{\p \in [0,p_M]}|G'(\p)| \int_0^t\int_{\Omega}\n \left|\partialt \p\right|\dx{x}\dx{t} \leq C(T),
\end{equation*}
thanks to Eq. \eqref{eq:L1n}. Then, it holds
\begin{equation*}
	\left\|\partialt \p\right\|_{L^1(\OmegaT)}\leq \iint_{\OmegaT \cap \{\n\leq 1/2\}}\gamma \n^{\gamma-1}\left|\partialt \n\right|\dx{x}\dx{t} + 2 \iint_{\OmegaT \cap \{\n > 1/2\}} \n \left|\partialt \p\right|\dx{x}\dx{t}\leq C(T).
\end{equation*}
The same argument can be used for the space derivatives of $\p$ and it goes through without major changes.
	 
We can actually gain more information on the pressure gradient, by integrating Eq. \eqref{eq:p} in space, \ie, 
\begin{align*}
     \int_{\Omega} \partialt {\p}\dx{x}  = \gamma\int_{\Omega}\p \prt*{\Lap \prt*{\p+\Phi} + G(\p)}\dx{x} +\int_{\Omega} \nabla \p \cdot\grad (\p + \Phi)\dx{x}.
\end{align*}
Integration by parts yields
\begin{align*}
     \int_{\Omega} \partialt \p\dx{x} \leq(1- \gamma)\int_{\Omega}|\grad \p|^2\dx{x}  +\gamma \int_{\Omega} \p G(\p)\dx{x} +(1-\gamma)\int_{\Omega} \nabla \p \cdot\grad \Phi\dx{x},
\end{align*}
and using Young's inequality we obtain
\begin{align*}
     \frac{\gamma -1}{2}\iint_{\OmegaT}|\grad \p(t)|^2\dx{x}\dx{t}  \leq \|\p^0\|_{L^1(\Omega)}  + \frac{(\gamma-1)}{2}  \iint_{\OmegaT} |\nabla \Phi|^2\dx{x}\dx{t}+ \gamma \ \iint_{\OmegaT} |\p G(\p)|\dx{x}\dx{t}.
\end{align*}
Dividing by $(\gamma -1)$ we finally get
\begin{align*}
    \iint_{\OmegaT}|\grad \p|^2 \dx{x}\dx{t} \leq  C(T),
\end{align*}
which concludes the proof.
\end{proof}
\section{Stronger bounds on $\p$}\label{sec:strongerest} 
This section is dedicated to establishing more refined estimates on the pressure, \cf Lemma \ref{lemma:L4} and Lemma \ref{lemma:ABL3}. Upon obtaining those estimates we will then be able to proceed to proving the strong compactness of the pressure gradient, \cf Lemma \ref{lemma:strongconv}, which is crucial in the overall endeavour of establishing the incompressible limit.

The first result on the pressure's regularity is the $L^4$-boundedness of its gradient. This bound was already proved in \cite{MePeQu}, although, the authors use the $L^\infty$-version of the Aronson-B\'enilan estimate. Here we adapted the method used in \cite{DP}, where a new method was employed, that does not require any estimate on $\Delta \p$. 
\begin{lemma}[$L^4$-estimate of the pressure gradient.]\label{lemma:L4}
Given $T>0$, there exists a positive constant $C$, independent of $\gamma$, such that
\begin{equation*}  
   \iint_\OmegaT \p \sum_{i,j=1}^d \left|\frac{\partial^2  \p}{\partial x_i \partial x_j}\right|^2\dx{x}\dx{t} + (\gamma-1)\iint_\OmegaT  \p |\Delta \p+\Delta \Phi +G|^2\dx{x}\dx{t}\leq C(T),
\end{equation*}
as well as
\begin{equation*}
      \iint_\OmegaT |\nabla \p|^4 \dx{x}\dx{t} \leq C(T).
\end{equation*}
\end{lemma}
\begin{proof}
We write the equation for the pressure as follows
\begin{equation}
    \label{eq:pressure_rephrased}
    \partialt \p = \gamma \p (\Delta \f +G) + \nabla \p \cdot \nabla \f,
\end{equation}
where $\f:= \p + \Phi$.
We multiply Eq. \eqref{eq:pressure_rephrased}  by $-(\Delta \f +G)$ and integrate in space and time to obtain
\begin{align}\label{eq:total}
\begin{split}
    \int_0^T\ddt \int_\Omega &\frac{|\nabla \p|^2}{2}\dx{x}\dx{t} -\iint_\OmegaT \Delta \Phi \frac{\partial \p}{\partial t} \dx{x}\dx{t} - \iint_\OmegaT G \partialt \p \dx{x}\dx{t}\\[0.7em]
    &=\underbrace{- \iint_\OmegaT\nabla \p \cdot \nabla \f (\Delta \f +G)\dx{x}\dx{t} }_{\curlyI} - \gamma \iint_\OmegaT \p |\Delta \f + G|^2\dx{x}\dx{t}. 
\end{split}
\end{align}
For convenience, let us define the function $\overline{G}=\overline{G}(\p)=\int_0^{\p} G(q) \dx{q}$. Thus, we have 
$$\partial_t \p \,G(\p)=\partial_t \overline{G}(\p),
$$
and thus
\begin{equation*}
	\iint_\OmegaT \partialt \p \,G(\p)\dx{x}\dx{t} = \int_0^T\ddt \int_\Omega \overline{G}(\p)\dx{x}\dx{t}.
\end{equation*}
Now, we need to estimate the term $\curlyI$ on the right-hand side of Eq. \eqref{eq:total}.
Since $\p=\f-\Phi$ we have
\begin{align*}
    \curlyI=& - \iint_\OmegaT\nabla \p \cdot \nabla \f (\Delta \f +G)\dx{x}\dx{t} \\[0.7em]
    =& - \iint_\OmegaT |\nabla \f|^2 \Delta \f\dx{x}\dx{t} + \iint_\OmegaT\nabla \Phi \cdot \nabla \f \Delta \f\dx{x}\dx{t} - \iint_\OmegaT G \nabla \p\cdot \nabla \f\dx{x}\dx{t} \\[0.7em]
	\leq& - \underbrace{\iint_\OmegaT |\nabla \f|^2 \Delta \f\dx{x}\dx{t}}_{\curlyI_1}  + \underbrace{\iint_\OmegaT \nabla \Phi \cdot \nabla \f \Delta \f \dx{x}\dx{t}}_{\curlyI_2}+ C,
\end{align*}
thanks to the $L^2$-bounds of both $\nabla \p$ and $\nabla \Phi$.
We integrate by parts twice in space the term $\curlyI_1$ and obtain
\begin{align*}
	\curlyI_1 &=\iint_\OmegaT \f \Delta(|\nabla \f|^2)\dx{x}\dx{t}\\[0.7em]
	& = 2 \iint_\OmegaT \f \nabla \f \cdot \nabla (\Delta \f)\dx{x}\dx{t} + 2 \iint_\OmegaT \f \sum_{i,j=1}^d \left|\frac{\partial^2  f}{\partial x_i \partial x_j}\right|^2\dx{x}\dx{t}\\[0.7em]
	&= -2 \iint_\OmegaT\f |\Delta \f|^2\dx{x}\dx{t}- 2\iint_\OmegaT |\nabla \f|^2\Delta \f\dx{x}\dx{t} + 2 \iint_\OmegaT \f \sum_{i,j=1}^d \left|\frac{\partial^2  f}{\partial x_i \partial x_j}\right|^2\dx{x}\dx{t}.
\end{align*}
Hence, we obtain
\begin{equation*}
	-\curlyI_1=\iint_{\OmegaT} |\nabla \f|^2 \Delta \f\dx{x}\dx{t} = \underbrace{\frac{2}{3} \iint_{\OmegaT}\f |\Delta \f|^2\dx{x}\dx{t}}_{\curlyI_{1,1}} - \frac{2}{3} \iint_{\OmegaT}\f \sum_{i,j=1}^d \left|\frac{\partial^2  \f}{\partial x_i \partial x_j}\right|^2\dx{x}\dx{t}.
\end{equation*}
The first integral can be bounded by
\begin{align*}
    \curlyI_{1,1}=\frac 2 3 \iint_\OmegaT \f |\Delta \f|^2\dx{x}\dx{t} \leq& \ \frac 2 3 \iint_\OmegaT  \f |\Delta \f+G|^2\dx{x}\dx{t} - \frac 4 3 \iint_\OmegaT  \f \Delta \f G\dx{x}\dx{t}\\[0.7em]
    &\qquad - \frac 2 3 \iint_\OmegaT \f G^2\dx{x}\dx{t}\\[0.7em]
    \leq& \frac 2 3 \iint_\OmegaT  \f |\Delta \f+G|^2\dx{x}\dx{t} +C,
\end{align*}
since, by integration by parts,
\begin{equation*}
    - \iint_\OmegaT  \f \Delta \f G\dx{x}\dx{t} = \frac{4}{3}  \iint_\OmegaT  |\nabla\f|^2 G\dx{x}\dx{t} +\frac{4}{3}  \iint_\OmegaT G' \nabla\f\cdot\nabla\p\dx{x}\dx{t} \leq C.
\end{equation*}
Thus, we conclude
\begin{equation}\label{eq:I1}
   - \curlyI_1 \leq \frac 2 3 \iint_\OmegaT  \f |\Delta \f+G|^2\dx{x}\dx{t} - \frac{2}{3} \iint_{\OmegaT}\f\sum_{i,j=1}^d \left|\frac{\partial^2  \f}{\partial x_i \partial x_j}\right|^2\dx{x}\dx{t}+C
\end{equation}
Now we compute the term $\curlyI_2$
\begin{align*}
    \curlyI_2=\iint_\OmegaT \nabla \Phi \cdot \nabla \f \Delta \f\dx{x}\dx{t}= &-\iint_\OmegaT \Delta \Phi \Delta \f \f \dx{x}\dx{t} - \iint_\OmegaT \nabla \Phi \cdot \nabla(\Delta \f) \f \dx{x}\dx{t}\\[0.7em]
    = &-\iint_\OmegaT \Delta \Phi \Delta \f \f \dx{x}\dx{t} + \iint_\OmegaT \sum_{i,j=1}^d \frac{\partial^2 \Phi}{\partial x_i \partial x_j}  \frac{\partial^2 \f}{\partial x_i \partial x_j}  \f\dx{x}\dx{t} \\[0.7em]
    &+ \iint_\OmegaT \f \sum_{i,j=1}^d \frac{\partial \Phi}{\partial x_j}  \frac{\partial \f}{\partial x_i} \frac{\partial^2 \f}{\partial x_i \partial x_j}  \dx{x}\dx{t}.
\end{align*}
Let us notice that the last term can be written as
\begin{equation*}
    \iint_\OmegaT \sum_{i,j=1}^d \frac{\partial \Phi}{\partial x_j}  \frac{\partial \f}{\partial x_{i}} \frac{\partial^2 \f}{\partial x_i \partial x_j} \dx{x}\dx{t}= \frac 1 2\iint_\OmegaT \nabla\Phi \cdot \nabla (|\nabla \f|^2)\dx{x}\dx{t}.
\end{equation*}
Hence, using the identity
\begin{equation*}
  \frac 1 2\iint_\OmegaT \nabla\Phi \cdot \nabla (|\nabla \f|^2)\dx{x}\dx{t}=  
    \iint_\OmegaT \nabla \cdot \prt*{ \nabla\Phi |\nabla \f|^2} \dx{x}\dx{t} - \frac 1 2\iint_\OmegaT \Delta\Phi |\nabla \f|^2\dx{x}\dx{t},
\end{equation*}
we finally obtain
\begin{equation*}
    \curlyI_2 = -\underbrace{\frac 1 2\iint_\OmegaT \Delta \Phi|\nabla\f|^2\dx{x}\dx{t}}_{\curlyI_{2,1}} - \underbrace{\iint_\OmegaT \f \Delta \Phi \Delta \f \dx{x}\dx{t}}_{\curlyI_{2,2}} + \underbrace{\iint_\OmegaT\f \sum_{i,j=1}^d \frac{\partial^2 \Phi}{\partial x_i \partial x_j} \frac{\partial^2 \f}{\partial x_i \partial x_j}\dx{x}\dx{t}}_{\curlyI_{2,3}}.
\end{equation*}
The term $\curlyI_{2,1}$ is bounded since $\Delta \Phi\in L^\infty(\OmegaT)$ and $\nabla \f \in L^2(\OmegaT)$. 

We treat the term $\curlyI_{2,2}$ as follows
\begin{align*}
   \curlyI_{2,2} &= -\iint_\OmegaT \f \Delta \Phi \Delta \f\dx{x}\dx{t}\\[0.7em]
   &= -\iint_\OmegaT \f \Delta \Phi (\Delta \f +G) \dx{x}\dx{t}  + \iint_\OmegaT \f \Delta \Phi G\dx{x}\dx{t} \\[0.7em]
   &\leq \frac 1 3 \iint_\OmegaT \f |\Delta \f +G|^2\dx{x}\dx{t} + \frac{3}{4} \iint_\OmegaT \f |\Delta \Phi|^2\dx{x}\dx{t} +C\\[0.7em]
   &\leq \frac 1 3 \iint_\OmegaT  \f |\Delta \f +G|^2\dx{x}\dx{t} + C.
\end{align*}
Using Young's inequality, we bound the last integral
\begin{align*}
    \curlyI_{2,3} &= \iint_\OmegaT \f \sum_{i,j=1}^d \frac{\partial^2  \Phi}{\partial x_i \partial x_j} \frac{\partial^2  f_\gamma}{\partial x_i \partial x_j} \dx{x}\dx{t}\\[0.7em]
    &\leq \ \frac{2}{3}\iint_\OmegaT \Phi \sum_{i,j=1}^d \left|\f \frac{\partial^2 \Phi}{\partial x_i \partial x_j}\right|^2 \dx{x}\dx{t} + \frac 3 8 \iint_\OmegaT \Phi \sum_{i,j=1}^d \left|\frac{\partial^2 \Phi}{\partial x_i \partial x_j} \right|^2\dx{x} \dx{t} \\[0.7em]
    &\qquad + \frac 1 2 \iint_\OmegaT \p \sum_{i,j=1}^d \left|\frac{\partial^2\f}{\partial x_i \partial x_j}\right|^2\dx{x} \dx{t} +\frac 1 2 \iint_\OmegaT \p \sum_{i,j=1}^d \left|\frac{\partial^2 \Phi}{\partial x_i \partial x_j} \right|^2\dx{x}\dx{t}\\[0.7em]
    &\leq  \ \frac{2}{3}\iint_\Omega \Phi \sum_{i,j=1}^d \left|\frac{\partial^2\f}{\partial x_i \partial x_j}\right|^2 \dx{x}\dx{t}
    +\frac 1 2 \iint_\OmegaT \p \sum_{i,j=1}^d\left|\frac{\partial^2\f}{\partial x_i \partial x_j}\right|^2 \dx{x} \dx{t} + C.
\end{align*}
Thus, we finally have
\begin{align}
    \label{eq:I2}
    \begin{split}
    \curlyI_2 &\leq \frac 1 3 \iint_\OmegaT  \f |\Delta \f +G|^2\dx{x}\dx{t} +\frac{2}{3}\iint_\OmegaT \Phi \sum_{i,j=1}^d\left|\frac{\partial^2\f}{\partial x_i \partial x_j}\right|^2 \dx{x}\dx{t}\\[0.7em]
    &\qquad + \frac12 \iint_\OmegaT \p \sum_{i,j=1}^d\left|\frac{\partial^2\f}{\partial x_i \partial x_j}\right|^2 \dx{x}\dx{t} +C.
    \end{split}
\end{align}
Gathering Eq. \eqref{eq:I1} and Eq. \eqref{eq:I2}, we get
\begin{align*}
  \curlyI
  &\leq -\curlyI_1 + \curlyI_2 + C\\[0.7em]
  &\leq \iint_\OmegaT \f |\Delta \f+G|^2\dx{x}\dx{t} -\frac 1 6 \int_\Omega \p \sum_{i,j=1}^d \left|\frac{\partial^2\f}{\partial x_i \partial x_j}\right|^2\dx{x} \dx{t} +C\\[0.7em]
  &=\iint_\OmegaT \Phi |\Delta \f+G|^2\dx{x}\dx{t}+ \iint_\OmegaT \p |\Delta \f+G|^2 \dx{x}\dx{t} -\frac 1 6 \iint_\OmegaT \p \sum_{i,j=1}^d \left|\frac{\partial^2\f}{\partial x_i \partial x_j}\right|^2\dx{x}\dx{t}  +C.
\end{align*}
Since we are considering compactly supported solutions and since $\Phi$ is bounded, we can assume it to be negative, substituting $\Phi$ by $\Phi- \|\Phi\|_{L^\infty}$, in the equations for $\p$. Then, the first integral in the previous equation is negative.

Gathering all the bounds we can write Eq. \eqref{eq:total} as
\begin{align*} 
\begin{split}
  \frac 1 6 \iint_\OmegaT &\p \sum_{i,j=1}^d\left|\frac{\partial^2\f}{\partial x_i \partial x_j}\right|^2\dx{x}\dx{t} +\left(\gamma-1\right) \iint_\Omega  \p |\Delta \f+G|^2\dx{x}\dx{t} \\[0.7em]
  &\leq \int_0^T \ddt \int_\Omega \left( \overline{G}-\frac{|\nabla \p|^2}{2}\right) \dx{x} \dx{t} + \iint_\OmegaT \Delta \Phi \partialt{\p} \dx{x}\dx{t}+ C\\[0.7em]
  &\leq \ C(T),
  \end{split}
\end{align*}
where in the last inequality we used the $L^1$-bound of $\partial_t \p$.
Thus, we have proved the following bound
\begin{align*}
    \frac16  \iint_\OmegaT \p \sum_{i,j=1}^d\left|\frac{\partial^2\f}{\partial x_i \partial x_j}\right|^2\dx{x}\dx{t} +\left(\gamma-1\right) \iint_\OmegaT \p |\Delta \f+G|^2 \dx{x}\dx{t} \leq C(T).
\end{align*}
Finally, thanks to the boundedness of $\partial_{i,j}^2 \Phi$, we have
\begin{align}
\label{eq: pd2p}
\begin{split}
    \iint_\OmegaT &\p \sum_{i,j=1}^d \left|\frac{\partial^2 \p}{\partial x_i \partial x_j}\right|^2 \dx{x}\dx{t}\\[0.7em]
    &\leq 2\iint_\OmegaT \p \sum_{i,j=1}^d\left|\frac{\partial^2\f}{\partial x_i \partial x_j}\right|^2\dx{x}\dx{t} +2\iint_\OmegaT \p \sum_{i,j=1}^d \left|\frac{\partial^2\Phi}{\partial x_i \partial x_j}\right|^2\dx{x}\dx{t}\\[0.7em]
    &\leq C(T),
\end{split}
\end{align}
and since $\gamma>1$
\begin{align*}
   \iint_\OmegaT \p |\Delta \p+G|^2\dx{x}\dx{t} 
   &\leq 2\iint_\OmegaT  \p |\Delta \f+G|^2\dx{x}\dx{t} +2\iint_\OmegaT  \p |\Delta \Phi+G|^2\dx{x}\dx{t}\\[0.7em]
   &\leq C(T),
\end{align*}
and the first part of the lemma is proven. Now it remains to prove the $L^4$-bound of the pressure gradient. Integrating by parts we have
\begin{align*}
	\int_{\Omega} |\nabla \p|^4\dx{x} = - \int_\Omega \p \Delta \p |\nabla \p|^2\dx{x} - \int_\Omega \p \nabla \p\cdot \nabla(|\nabla \p|^2)\dx{x}.
\end{align*}
Applying Young's inequality to the first term, we obtain
\begin{align*}
	\frac{1}{2} \int_{\Omega} |\nabla \p|^4\dx{x} \leq \frac{1}{2} \int_{\Omega} \p^2 |\Delta \p|^2\dx{x} - 2 \sum_{i,j=1}^d  \int_{\Omega}\p \frac{\partial \p}{\partial x_i} \frac{\partial\p}{\partial x_j} \frac{\partial^2 \p}{\partial x_i \partial x_j}\dx{x}.
\end{align*}
Thanks to Young's inequality, the last term can be bounded from above by 
\begin{align*}
	\left| 2 \sum_{i,j=1}^d  \int_{\Omega}\p \frac{\partial \p}{\partial x_i} \frac{\partial\p}{\partial x_j} \frac{\partial^2 \p}{\partial x_i \partial x_j} \dx{x}\right| \leq \frac{1}{4} \int_{\Omega} |\nabla \p|^4\dx{x} + 4\int_{\Omega} \p^2 \sum_{i,j=1}^d \left|\frac{\partial^2  \p}{\partial x_i \partial x_j}\right|^2\dx{x}.  
\end{align*}
Therefore, we obtain
\begin{equation*}
	\frac{1}{4} \int_{\Omega} |\nabla \p|^4\dx{x} \leq \frac{1}{2} \int_{\Omega} \p^2 |\Delta \p|^2 \dx{x} + 4\int_{\Omega} \p^2 \sum_{i,j=1}^d \left|\frac{\partial^2  \p}{\partial x_i \partial x_j}\right|^2\dx{x}.
\end{equation*}
Since $\p\leq p_M$ and thanks to Eq. \eqref{eq: pd2p}, we conclude that
\begin{align*}
	\iint_{\OmegaT} |\nabla \p|^4\dx{x}\dx{t} \leq& C(T),
\end{align*}
which completes the proof.
\end{proof}

Building on the $L^4$-estimate on the pressure gradient, we are now dedicated to an additional bound on the pressure which, by itself, yields $L^1$-compactness of the pressure gradient. In conjunction with the $L^4$-estimate the gradient is then shown to be strongly compact in any $L^p(\OmegaT)$, for $1\leq p<4$, \cf Lemma \ref{lemma:strongconv}. The subsequent estimate is an $L^p$-version of the celebrated Aronson-B\'enilan estimate, \cf \cite{AB, BPS2020}. At the heart of its proof is the study of an auxiliary second-order quantity and its evolution along the flow of the pressure equation. We define $w:=\Delta \p +G(\p)$ and, for the reader's convenience, recall that the pressure satisfies the equation
\begin{equation}
    \label{eq:pressure_for_AB}
    \partialt \p= \gamma \p w + \gamma \p \Delta \Phi + \nabla \p \cdot (\nabla \p + \nabla \Phi).
\end{equation} 
\begin{lemma}[Aronson-B\'enilan $L^3$-estimate.]\label{lemma:ABL3}
For all $T>0$ and $\gamma > \max(1, 2-\frac 2 d) $, there exists a positive constant $C(T)$, independent of $\gamma$, such that 
\begin{equation*}
    \iint_{\OmegaT} |w|_-^3 \dx{x}\dx{t}\leq C(T).
\end{equation*}
\end{lemma}
\begin{proof}
We compute the time derivative of $w$
\begin{align*}
    \partialt w = &\gamma \Delta (\p w) + \gamma \p \Delta (\Delta \Phi) + \gamma (w-G)\Delta \Phi + 2 \gamma \nabla \p \cdot \nabla (\Delta \Phi) + 2 \nabla \p \cdot \nabla (w-G) \\
    &+ 2 \sum_{i,j=1}^d  \left|\frac{\partial^2 \p}{\partial x_i \partial x_j}\right|^2 +\nabla (w-G)\cdot \nabla \Phi + \nabla \p \cdot \nabla (\Delta \Phi) + 2 \sum_{i,j=1}^d  \frac{\partial^2 \p}{\partial x_i \partial x_j} \frac{\partial^2 \Phi}{\partial x_i \partial x_j} + G' \partialt \p. 
\end{align*}
Young's inequality yields
\begin{align*}
    \left| 2 \sum_{i,j=1}^d  \partial_{i,j}^2\p \frac{\partial^2 \Phi}{\partial x_i \partial x_j} \right| \leq \sum_{i,j=1}^d \left|\frac{\partial^2 \p}{\partial x_i \partial x_j}\right|^2 +\sum_{i,j=1}^d \left|\frac{\partial^2 \Phi}{\partial x_i \partial x_j}\right|^2,
\end{align*}
and thus, using Eq. \eqref{eq:pressure_for_AB}, we get
\begin{align*}
    \partialt w \geq &\gamma \Delta (\p w) + \gamma \p \Delta (\Delta \Phi) +\gamma w \Delta \Phi - \gamma G\Delta \Phi + (2 \gamma +1) \nabla \p \cdot \nabla (\Delta \Phi) + 2 \nabla \p \cdot \nabla w\\[0.7em]
    &-2|\nabla p|^2 G' + \sum_{i,j=1}^d  \left|\frac{\partial^2 \p}{\partial x_i \partial x_j}\right|^2 -\sum_{i,j=1}^d \left|\frac{\partial^2 \Phi}{\partial x_i \partial x_j}\right|^2 +\nabla w\cdot \nabla \Phi- G'\nabla p \cdot \nabla \Phi\\[0.5em]
    &+ \gamma G' \p w +\gamma \p G' \Delta \Phi + G' |\nabla \p|^2 +G' \nabla \p \cdot \nabla \Phi .
\end{align*}
We use the fact that 
\begin{equation*}
    \sum_{i,j=1}^d  \left|\frac{\partial^2 \p}{\partial x_i \partial x_j}\right|^2\geq \frac
    1 d |\Delta \p|^2 = \frac{1}{d} (w-G)^2,
\end{equation*}
and we obtain
\begin{align*}
    \partialt w \geq &\gamma \Delta (\p w) + \gamma \p \Delta (\Delta \Phi) + \gamma w \Delta \Phi- \gamma G\Delta \Phi + (2 \gamma +1) \nabla \p \cdot \nabla (\Delta \Phi) + 2 \nabla \p \cdot \nabla w\\ 
    &-|\nabla p|^2 G' 
    +\frac 1 d w^2 - \frac{2}{d} w G + \frac{1}{d} G^2 -\sum_{i,j=1}^d \left|\frac{\partial^2 \Phi}{\partial x_i \partial x_j}\right|^2  +\nabla w\cdot \nabla \Phi
    \\
    &+ \gamma G' \p w +\gamma \p G' \Delta \Phi.
\end{align*}
We multiply by $-|w|_-$, to find
\begin{align*}
    -\partialt w |w|_- \leq &- \frac 1 d |w|_-^3 + \gamma \Delta \Phi |w|_-^2 -\frac{2}{d}G |w|_-^2+ \gamma G' \p |w|_-^2
    -\frac{1}{d}G^2|w|_- +\gamma G\Delta \Phi|w|_- \\[0.7em]
    &+\sum_{i,j=1}^d \left|\frac{\partial^2 \Phi}{\partial x_i \partial x_j}\right|^2|w|_- -\gamma \p G' \Delta \Phi |w|_- +|\nabla \p|^2 G'|w|_- \\[0.6em]
    &+\gamma \Delta (\p |w|_-) |w|_-+ 2 \nabla \p \cdot \nabla |w|_-|w|_-\\[1em]
    &-\gamma \p \Delta (\Delta \Phi)|w|_- - (2 \gamma +1) \nabla \p \cdot \nabla (\Delta \Phi)|w|_-  \\[1em]
     &+\nabla \Phi\cdot \nabla |w|_- |w|_- .
\end{align*}
Hence, using the fact that $G'<-\alpha$ and integrating in space and time, we obtain
\begin{align}\label{eq:finalinequality}
\begin{split}
    - \int_\Omega \frac{|w^0|_-^2}{2}\dx{x}\leq &- \frac 1 d  \iint_\OmegaT |w|_-^3 \dx{x}\dx{t} +C \gamma   \iint_\OmegaT |w|^2_-\dx{x}\dx{t} +C \gamma \iint_\OmegaT |w|_-\dx{x}\dx{t}
    \\[1em]
    &
    +\underbrace{\gamma \iint_\OmegaT \Delta (\p |w|_-) |w|_- + 2  \nabla \p \cdot \nabla |w|_-|w|_-\dx{x}\dx{t}}_{\curlyI_1}\\
    &\underbrace{-\gamma  \iint_\OmegaT \p \Delta (\Delta \Phi)|w|_-\dx{x}\dx{t}}_{\curlyI_2} - \underbrace{(2 \gamma +1) \iint_\OmegaT \nabla \p \cdot \nabla (\Delta \Phi)|w|_-\dx{x}\dx{t}}_{\curlyI_3} \\ 
    &+\underbrace{ \iint_\OmegaT \nabla \Phi\cdot \nabla |w|_- |w|_-\dx{x}\dx{t}}_{\curlyI_4} 
    \end{split}
\end{align}
where $C$ represents different constants depending on the $L^\infty$-norms of $G$, $G'$ and $\partial^2_{i,j} \Phi$, for $i,j=1,\dots,d$.

Now, we compute each term individually. Integration by parts yields
\begin{align*}
    \curlyI_1= &\gamma \iint_\OmegaT \Delta (\p |w|_-) |w|_- + 2  \nabla \p \cdot \nabla |w|_-|w|_-\dx{x}\dx{t}\\[0.7em]
    = &-\frac \gamma 2  \iint_\OmegaT \nabla \p \cdot \nabla |w|_-^2\dx{x}\dx{t} -\gamma  \iint_\OmegaT p \left|\nabla |w|_-\right|^2\dx{x}\dx{t} +  \iint_\OmegaT \nabla \p \cdot \nabla |w|_-^2\dx{x}\dx{t}\\[0.7em]
    = & -\left(1-\frac \gamma 2\right)  \iint_\OmegaT(w-G)|w|_-^2\dx{x}\dx{t} -\gamma \iint_\OmegaT \p \left|\nabla |w|_-\right|^2\dx{x}\dx{t} \\[0.7em]
    = & \left(1-\frac \gamma 2\right) \iint_\OmegaT|w|_-^3\dx{x}\dx{t} + \left(1-\frac \gamma 2\right) \iint_\OmegaT G |w|_-^2\dx{x}\dx{t} - \gamma  \iint_\OmegaT \p \left|\nabla |w|_-\right|^2\dx{x}\dx{t}\\[0.7em]
    \leq &  \left(1-\frac \gamma 2\right) \iint_\OmegaT|w|_-^3\dx{x}\dx{t} - \gamma  \iint_\OmegaT \p \left|\nabla |w|_-\right|^2\dx{x}\dx{t} + C \gamma  \iint_\OmegaT |w|_-^2\dx{x}\dx{t}.
\end{align*}
We continue by using integration by parts and Young's inequality to get
\begin{align*}
    \curlyI_2=&-\gamma  \iint_\OmegaT \p \Delta (\Delta \Phi)|w|_-\dx{x}\dx{t}\\[0.7em]
    =&\gamma \iint_\OmegaT \p \nabla(\Delta \Phi)\cdot\nabla|w|_-\dx{x}\dx{t} + \gamma \iint_\OmegaT \nabla \p \cdot \nabla(\Delta \Phi) |w|_- \dx{x}\dx{t}\\[0.7em]
    \leq &\frac \gamma 2 \iint_\OmegaT \p\left|\nabla|w|_-\right|^2\dx{x}\dx{t} +\frac \gamma 2  \iint_\OmegaT\p \left|\nabla (\Delta \Phi)\right|^2\dx{x}\dx{t} \\[0.7em]
    &+ \gamma \left(\iint_\OmegaT|\nabla\p|^4 \right)^{1/4}\left(\iint_\OmegaT |\nabla(\Delta \Phi)|w|_-|^{4/3}\dx{x}\dx{t}\right)^{3/4}\\[0.7em]
    \leq & \frac \gamma 2 \iint_\OmegaT \p\left|\nabla|w|_-\right|^2\dx{x}\dx{t} +\frac \gamma 2  \iint_\OmegaT\p \left|\nabla (\Delta \Phi)\right|^2\dx{x}\dx{t} \\[0.7em]
    &+ C \gamma \left(\iint_\OmegaT |\nabla(\Delta \Phi)|^{12/5}\dx{x}\dx{t}\right)^{5/12}     \left(\iint_\OmegaT|w|_-^{3}\dx{x}\dx{t} \right)^{1/3}\\[0.7em]
    \leq &\frac \gamma 2 \iint_\OmegaT \p\left|\nabla|w|_-\right|^2\dx{x}\dx{t} + C \gamma + C \gamma \left(\iint_\OmegaT|w|_-^{3}\dx{x}\dx{t} \right)^{1/3},
\end{align*}
where we used H\"older's inequality, the $L^4$-bound of the pressure gradient of Lemma~\ref{lemma:L4} and the assumption \eqref{eq:assptn_phi_AB},  $\nabla(\Delta \Phi) \in L_{\mathrm{loc}}^{12/5}(Q_T)$.

Using again Young's and Holder's inequalities we have
\begin{align*}
    \curlyI_3\leq &(2\gamma+1) \left(\iint_\OmegaT|\nabla\p|^4\dx{x}\dx{t} \right)^{1/4}\left(\iint_\OmegaT |\nabla(\Delta \Phi)|w|_-|^{4/3}\dx{x}\dx{t}\right)^{3/4}\\[0.7em] 
    \leq & C \gamma \left(\iint_\OmegaT |\nabla(\Delta \Phi)|^{12/5}\dx{x}\dx{t}\right)^{5/12}
    \left(\iint_\OmegaT|w|_-^{3}\dx{x}\dx{t} \right)^{1/3}\\[0.7em]
    \leq &C \gamma  \left(\iint_\OmegaT |w|_-^3 \dx{x}\dx{t}\right)^{1/3}.
\end{align*}
The last term is
\begin{equation*}
    \curlyI_4=\iint_\OmegaT \frac 1 2\nabla \Phi \cdot \nabla |w|_-^2\dx{x}\dx{t} = -\frac 1 2 \iint_\OmegaT  \Delta \Phi |w|_-^2\dx{x}\dx{t}\leq C \iint_\OmegaT |w|_-^2\dx{x}\dx{t}.
\end{equation*}
Here we have used the fact that $\Omega$ is a compact set which contains $\text{supp}(\p)$ and large enough such that $\Delta \p=0$ on $\partial \Omega$, then $|w|_-=0$ on $\partial \Omega$.

Hence, gathering all the estimates and using H\"older's inequality, we can rewrite Eq. \eqref{eq:finalinequality} as 
\begin{align*}
    \left(\frac \gamma 2 -1 +\frac 1 d\right) \iint_\OmegaT |w|_-^3\dx{x}\dx{t} \leq& \ C \gamma \left(\iint_\OmegaT |w|_-^3\dx{x}\dx{t}\right)^{1/3} +C \gamma \left(\iint_\OmegaT |w|_-^3\dx{x}\dx{t}\right)^{2/3}+ C \gamma ,
\end{align*}
since we assumed $|w^0|_- \in L^2(\R^d)$.
Finally, for $\gamma > \max(1, 2 - 2/d)$, we have
\begin{equation*}
    \iint_\OmegaT |w|_-^3 \dx{x}\dx{t}\leq C \left(\iint_\OmegaT |w|_-^3 \dx{x}\dx{t}\right)^{1/3}  + C \left( \iint_\OmegaT |w|_-^3 \dx{x}\dx{t}\right)^{2/3} + C,
\end{equation*}
which yields
\begin{equation*}
    \iint_\OmegaT |w|_-^3 \dx{x}\dx{t} \leq C(T),
\end{equation*}
where $C(T)$ depends on $T$, $|\Omega|$ and previous uniform bounds, and the proof is concluded.

\end{proof}
\begin{corollary}\label{corollary:L1laplacep}
It holds
\begin{equation}\label{eq:L1laplacep}
     \iint_{\OmegaT} |\Delta \p| \dx{x}\dx{t} \leq C(T).
\end{equation}
\end{corollary}
\begin{proof}
The compact support assumption yields
\begin{equation*}
    \iint_\OmegaT (\Delta \p + G) \dx{x}\dx{t} \leq C(T),
\end{equation*}
and then, thanks to H\"older's inequality, we have
\begin{align*}
    \iint_\OmegaT |\Delta \p+G| \dx{x}\dx{t} &= \iint_\OmegaT (\Delta \p +G) \dx{x}\dx{t} + 2 \iint_\OmegaT |w|_-\dx{x}\dx{t}\\
    &\leq C(T) + C \left(\iint_\OmegaT |w|^3_- \dx{x}\dx{t}\right)^{1/3}\\
    &\leq C(T).
\end{align*}
Finally, since $G$ is bounded, we obtain
\begin{equation*}
     \iint_\OmegaT |\Delta \p|\dx{x}\dx{t} \leq C(T).
\end{equation*}
\end{proof}

\begin{remark}
The proof of the Aronson-B\'enilan estimate can be made independent of the $L^4$-bound on $\nabla \p$ imposing a stronger condition on $\Phi$, namely $\nabla(\Delta\Phi)\in L^6$ rather than $L^{12/5}$.
\end{remark}

The bounds provided by Lemma~\ref{lemma:L4} and Lemma~\ref{lemma:ABL3} allow us to prove the strong convergence of $\nabla \p$ in $L^2(Q_T)$ thanks to compactness arguments, in particular the Fr\'echet-Kolmogorov theorem and the Aubin-Lions lemma.
\begin{lemma}[Strong convergence of the pressure gradient]\label{lemma:strongconv}
For any $T>0$ it holds
\begin{equation*}
    \nabla \p \rightarrow \nabla p_\infty,
\end{equation*}
strongly in $L^2(Q_T)$.
\end{lemma}
\begin{proof}
Thanks to Lemma~\ref{lemma:L4}, we infer the weak convergence (up to a subsequence) of the pressure gradient
\begin{equation}\label{eq:w4}
	\nabla \p \rightharpoonup \nabla p_\infty,
\end{equation}
weakly in $L^4(Q_T)$.
From Lemma~\ref{lemma:ABL3}, we know that $\Delta \p$ is bounded in $L^1(Q_T)$, thus, we recover the local compactness in space of the pressure gradient in any $L^r(Q_T)$, with $1\leq r < 4$. The proof of this claim is an extension of  \cite[Theorem~1]{BG} to a space-time setting.
	
To this end, let us take define the continuous function $\psi$, by setting
\begin{equation*}
    \begin{cases}
	    \psi(s)=-\epsilon, \quad &\text{ for } s<-\epsilon,\\
	    \psi(s)=s,\quad &\text{ for } -\epsilon\leq x \leq\epsilon,\\
	    \psi(s)=\epsilon, \quad &\text{ for } s>\epsilon,
	\end{cases}
\end{equation*} 
for $\epsilon>0$.
Given $\gamma, \hat\gamma >1$, we compute
\begin{equation*}
    \iint_\OmegaT |\nabla \p-\nabla p_{\hat\gamma}|^2 \psi'(\p-p_{\hat\gamma})\dx{x}\dx{t}= - \iint_\OmegaT (\Delta \p -\Delta p_{\hat\gamma})\psi(\p-p_{\hat\gamma})\dx{x}\dx{t}.
\end{equation*}
Next we split the domain into two parts by defining the set 
$$
    \Omega_{T,\epsilon}:=\{(x,t) \in \OmegaT \;| \; |\p(x,t) - p_{\hat\gamma}(x,t)|\leq\epsilon\}.
$$ 
Thus, since $\Delta \p$ is bounded in $L^1(Q_T)$ (uniformly with respect to $\gamma$), we have
\begin{equation*}
    \iint_{\Omega_{T,\epsilon}} |\nabla \p-\nabla p_{\hat\gamma}|^2\dx{x}\dx{t} \leq C \epsilon.
\end{equation*}
Hence
\begin{align*}
    \iint_{\Omega_{T}} |\nabla \p-\nabla p_{\hat\gamma}|\dx{x}\dx{t} = & \iint_{\Omega_{T,\epsilon}} |\nabla \p-\nabla p_{\hat\gamma}|\dx{x}\dx{t} + \iint_{{\Omega_{T,\epsilon}^c}} |\nabla \p-\nabla p_{\hat\gamma}|\dx{x}\dx{t}\\[0.7em]
    \leq&\ C \epsilon^{1/2} + 2 \ T^{1/2} \|\nabla \p\|_{L^2(Q_T)} \cdot |\Omega^c_{T,\epsilon}|^{1/2},
\end{align*}
where in the last line we used H\"older's inequality. Since $\p$ is compact, it is a Cauchy sequence, and there exist $\Gamma(\epsilon)$ large enough such that
for $\gamma,\hat\gamma>\Gamma(\epsilon)$ there holds
\begin{equation*}
    \iint_\OmegaT |\nabla \p-\nabla p_{\hat\gamma}|\dx{x}\dx{t}\leq C\epsilon^{1/2}+ C \epsilon.
\end{equation*}
This implies that $\nabla \p$ is a Cauchy sequence in $L^1(Q_T)$. Up to a subsequence we have a.e. convergence. Thanks to Eq. \eqref{eq:w4}, the pressure gradient is compact in space in any $L^r(Q_T)$, for $1\leq r <4$.
    
It now remains to prove time compactness. Since the pressure gradient converges in space we know, by the Fr\'echet-Kolmogorov theorem, that the space shifts converge, too. This is to say, there exists a function $\omega: \R\rightarrow\R_{\geq0}$, such that
\begin{equation}\label{eq:spaceshift}
    \iint_{\OmegaT}|\nabla \p(x+k,t)-\nabla \p(x,t)|\dx{x}\dx{t}\leq \omega(|k|),
\end{equation}
with $\omega(|k|)\rightarrow 0$ as $|k|\rightarrow 0$.
From Lemma~\ref{lemma:aprioriestimates}, $\partial_t \p$ is bounded in $L^1(Q_T)$.
Then, thanks to the Aubin-Lions lemma, we obtain the compactness in time by the following argument: First, let  $\phi_\alpha(x)=\alpha^{-d}\phi(\alpha^{-1} x)$, with $\alpha>0$ and $\phi\in C^\infty(\R^d)$  be non-negative, smooth mollifier  such that $\int_{\R^d} \phi(x)\dx{x}=1$. We compute the time shift of $\nabla \p$ with $h>0$, namely
\begin{align*}
    \int_0^{T-h} &\|\nabla \p(t+h)-\nabla \p(t)\|_{L^1(\Omega)}\dx{t}\\
    &\leq  \int_0^{T-h}\|\nabla \p(t+h)-\nabla \p(t+h)\star\phi_\alpha\|_{L^1(\Omega)}\dx{t}\\[0.7em]
    &\qquad+ \int_0^{T-h}\|(\nabla \p(t+h)-\nabla \p(t))\star\phi_\alpha\|_{L^1(\Omega)}\dx{t}\\[0.7em]
    &\qquad+ \int_0^{T-h}\|\nabla \p(t) \star \phi_\alpha-\nabla \p(t)\|_{L^1(\Omega)}\dx{t},
\end{align*}
having used the triangular inequality. Upon integration by parts, we have
\begin{align*}
    \int_0^{T-h}&\|(\nabla \p(t+h)-\nabla \p(t)) \star \phi_\alpha\|_{L^1(\Omega)}\dx{t}\\[0.7em]
    &=\int_0^{T-h}\left\|\int_0^h \left(\partialt \p (t+s) \star \nabla \phi_\alpha\right)\dx{s}\right\|_{L^1(\Omega)}\dx{t}\\[0.7em]
    &\leq \int_0^{T-h} \int_0^h \int_{\Omega} \left|\frac{\partial \p}{\partial t}(t+s) \star \nabla \phi_\alpha\right|\dx{x} \dx{s}\dx{t}\\[0.7em]
    &\leq \|\nabla \phi_\alpha\|_{L^1(\R^d)} \int_0^{T-h} \int_0^h\int_{\Omega}\left|\frac{\partial \p}{\partial t}(x, t+s) \right|\dx{x}\dx{s}\dx{t},
\end{align*}
having used Young's convolution inequality in the last line. Applying Fubini's theorem we obtain
\begin{align*}
    \int_0^{T-h}&\|(\nabla \p(t+h)-\nabla \p(t)) \star \phi_\alpha\|_{L^1(\Omega)}\dx{t}\\[0.7em]
    &\leq \|\nabla \phi_\alpha\|_{L^1(\R^d)} \int_0^{T-h} \int_0^h\int_{\Omega}\left|\frac{\partial \p}{\partial t}(x, t+s) \right|\dx{x}\dx{s}\dx{t}\\[0.7em]
    &= \|\nabla \phi_\alpha\|_{L^1(\R^d)} \int_0^{T-h} \int_t^{t+h}\int_{\Omega}\left|\frac{\partial \p}{\partial t}(x, s) \right|\dx{x}\dx{s}\dx{t}\\[0.7em]
    &=\|\nabla \phi_\alpha\|_{L^1(\R^d)} \int_{0}^{T-h} \int_{\max(0, s-h)}^{\min(s, T-h)} \int_{\Omega}\left|\frac{\partial \p}{\partial t}(x, s) \right|\dx{x}\dx{t}\dx{s}\\[0.7em]
    &\leq C\alpha^{-1}|h|,
\end{align*}
having used the assumption on the mollifier and the uniform $L^1$-bound on $\partial \p$.

Since the space shifts satisfy  condition \eqref{eq:spaceshift}, we obtain
\begin{align*}
    \int_0^{T-h}\|\nabla \p(t)\star\phi_\alpha-\nabla \p(t)\|_{L^1(\Omega)}\dx{t} \leq
    & \int_0^{T-h}\left\|\int_{\R^d}\phi(y)( \nabla \p(\cdot-\alpha y,t)-\nabla \p(\cdot,t)) \dx{y} \right\|_{L^1(\Omega)}\dx{t} \\
    \leq& \int_{\R^d}\phi(y)\iint_\OmegaT \left|\nabla \p(x-\alpha y,t)-\nabla \p(x,t)\right| \dx{x}\dx{t}\dx{y}\\ 
    \leq& \int_{\R^d}\phi(y)\omega(\alpha |y|) \dx{y}.
\end{align*}
Gathering all the estimates we have 
\begin{equation*}
    \int_0^{T-h}\|\nabla \p(x,t+h)-\nabla \p(x,t)\|_{L^1(\R^d)}\dx{t} \leq C\frac{|h|}{\alpha} + 2 \int_{\R^d}\phi(y)\omega(\alpha |y|) \dx{y}\dx{t}.
\end{equation*}
Taking $\alpha=|h|^{1/2}$, we conclude that
\begin{equation*}
    \int_0^{T-h}\|\nabla \p (x,t+h)-\nabla \p(x,t)\|_{L^1(\R^d)}\dx{t}\rightarrow 0,
\end{equation*}
as $|h|\rightarrow 0$, which ultimately yields the time compactness. Hence
\begin{equation*}
	\nabla \p \rightarrow \nabla p_\infty,
\end{equation*}
strongly in $L^q(Q_T)$, for $1 \leq q < 4$, and in particular for $q=2$.
\end{proof}

\begin{remark}
The tumour growth rate usually depends also on the presence of nutrients, therefore one can couple Eq. \eqref{eq:n}, with an equation on the nutrient concentration. Then, the model reads
\begin{equation}\label{eq: withnutri}
   \left\{
    \begin{array}{rl}
      \partialt{\n} - \nabla\cdot \prt*{\n \nabla \p} -\nabla\cdot\prt*{\n \nabla \Phi} \!\!\! &= \n G(\p,c_\gamma), \\[1em]
        \partialt{c_\gamma} - \Delta c_\gamma  \!\!\! &= - \n H(c_\gamma),
    \end{array}
    \right.
\end{equation}
where $H$ is the nutrient consumption rate. Thus, system \eqref{eq: withnutri} is actually an extension of the model with nutrient studied in \cite{PQV}. 

Let us notice that the proofs of the estimates in Lemma~\ref{lemma:L4} and Lemma~\ref{lemma:ABL3} can be adapted for system \eqref{eq: withnutri} without any particular difficulty. In fact, the boundedness of the new terms depending on $c_\gamma, \nabla c_\gamma$, and $ \Delta c_\gamma$ relies only on the $L^2$-regularity of $c_\gamma$ and its derivatives, which comes directly from its equation in system \eqref{eq: withnutri}. Therefore, the strong convergence stated in Lemma~\ref{lemma:strongconv} still holds for this model. We refer the reader to \cite{PQV} and \cite{DP} for the complete treatment of these additional terms. 
\end{remark}

\section{The Incompressible Limit}
\label{sec:incomp_limit}
The results obtained in Section~\ref{sec:strongerest} allow us to finally pass to the incompressible limit in Eq. \eqref{eq:p} and obtain the complementarity relation, Eq. \eqref{eq:cr}. Let us point out that, thanks to the uniform (with respect to $\gamma$) boundness of $\nabla \p$ in $L^2(Q_T)$ and $\partial_t \p$ in $L^1(Q_T)$, the complementarity relation turns out to be equivalent to the strong convergence of $\nabla \p$ in $L^2(Q_T)$, given by Lemma~\ref{lemma:strongconv}.
\begin{theorem}[Complementarity relation] \label{thm:CR}
We may pass to the limit in Eq. \eqref{eq:p}, as $\gamma \rightarrow \infty$, and obtain the so-called \textit{complementarity relation}
\begin{equation*}
    p_\infty (\Lap p_\infty + \Lap \Phi + G(p_\infty))=0,
\end{equation*}
in the distributional sense. Moreover, $n_\infty$ and $p_\infty$ satisfy the equations
\begin{subequations}
\begin{equation} \label{eq:limsys}
        \partialt{n_\infty} = \Delta p_\infty + n_\infty G(p_\infty) + \nabla \cdot (n_\infty \nabla \Phi),
\end{equation}
in $\curlyD'(Q_T)$, as well as
\begin{equation}
\label{eq:limsys2}
    p_\infty(1-n_\infty) = 0,
\end{equation}
almost everywhere.
\end{subequations}
\end{theorem}
\begin{proof}
Thanks to the bounds in Lemma~\ref{lemma:aprioriestimates},
\begin{equation*}
    \iint_\OmegaT \left|\partialt \p\right| +|\nabla \p|\dx{x}\dx{t} \leq C(T),
\end{equation*}
then, by the Fr\'echet-Kolmogorov Theorem, $\p$  is strongly compact in $L^1(Q_T)$, for all $T>0$.

We integrate Eq. \eqref{eq:p} against a test function $\varphi \in C^{\infty}_{\text{c}}(Q_T)$ to obtain
\begin{align*}
	\iint_{Q_T} \frac{\partial \p}{\partial t} \varphi\dx{x}\dx{t} = & (1-\gamma)\left( \iint_{Q_T}  |\nabla \p|^2\varphi\dx{x}\dx{t} + \iint_{Q_T} \nabla \p \cdot \nabla \Phi \varphi\dx{x}\dx{t} \right)\\
	&-\gamma \iint_{Q_T} \p \nabla \p \cdot \nabla \varphi\dx{x}\dx{t} -\gamma \iint_{Q_T} \p \nabla \Phi \cdot \nabla \varphi\dx{x}\dx{t}\\
	&+ \gamma \iint_{Q_T} \p G(\p)\varphi\dx{x}\dx{t} .
\end{align*}
Dividing by $\gamma-1$ and passing to the limit $\gamma \rightarrow \infty$, we obtain
\begin{align*}
	\lim_{\gamma\rightarrow\infty}&\left[-\iint_{Q_T} \left( |\nabla \p|^2\varphi + \p \nabla \p \cdot \nabla \varphi\right)\dx{x}\dx{t}\right. \\
	& \;\, \left. -\iint_{Q_T} \left(\nabla \p \cdot \nabla \Phi \varphi + \p \nabla \Phi \cdot \nabla \varphi \right)\dx{x}\dx{t} + \iint_{Q_T} \p G(\p)\varphi\dx{x}\dx{t}\right] =0.
\end{align*}
It remains to identify the limit.
By the strong convergence of $\p$ and $\nabla \p$ in $L^2(Q_T)$ we have
\begin{align*}
	&-\iint_{Q_T} \left( |\nabla p_\infty|^2\varphi+p_\infty \nabla p_\infty \cdot \nabla \varphi\right)\dx{x}\dx{t} -\iint_{Q_T} \left(\nabla p_\infty \cdot \nabla \Phi \varphi + p_\infty \nabla \Phi \cdot \nabla \varphi \right)\dx{x}\dx{t} \\
	&+ \iint_{Q_T} p_\infty G(p_\infty)\varphi \dx{x}\dx{t} = 0,
\end{align*}
\ie,
\begin{equation*}
	p_\infty (\Delta p_\infty + \Delta \Phi + G(p_\infty))=0,
\end{equation*}
in the distributional sense.
	
Now, we prove that Eq. \eqref{eq:limsys} and Eq. \eqref{eq:limsys2} are satisfied. By Lemma~\ref{lemma:aprioriestimates}, we have
\begin{equation*}
    \iint_\OmegaT \left|\partialt \n\right| +|\nabla \n|\dx{x}\dx{t}  \leq C(T),
\end{equation*}
and then we infer the compactness of the density. Up to a subsequence, we also have almost everywhere convergence, both for $\n$ and $\p$.
Passing to the limit in the relation $\p^{(1+\gamma)/\gamma}=\n \p$, we obtain 
\begin{equation*}
    p_\infty(1-n_\infty)=0,
\end{equation*}
a.e. in $Q_T$.
	
Now, we may pass to the limit in Eq. \eqref{eq:n} to obtain
\begin{equation*}
    \partialt{n_\infty} = \nabla\cdot \prt*{n_\infty \nabla p_\infty} + n_\infty G(p_\infty) + \nabla\cdot\prt*{n_\infty \nabla \Phi}. 
\end{equation*}
From the following relation 
\begin{equation*}
    \frac{1+\gamma}{\gamma}\n \nabla \p = \p \nabla \n+\n \nabla \p,
\end{equation*}
we infer $p_\infty \nabla n_\infty=0$, and thus
\begin{equation*}
    n_\infty \nabla p_\infty =\nabla p_\infty.
\end{equation*}
By consequence, $n_\infty$ and $p_\infty$ satisfy
\begin{equation*}
     \partialt{n_\infty} = \Delta p_\infty + n_\infty G(p_\infty) + \nabla\cdot\prt*{n_\infty \nabla \Phi},
\end{equation*}
which completes the proof.
\end{proof}
\section{Uniqueness of the Limit Pressure}
\label{sec:uniq_limit}
This section is dedicated to proving the following statement.
\begin{theorem}[Uniqueness of $n_\infty$ and $p_\infty$]
The incompressible limit obtained in the previous section, $(n_\infty, p_\infty)$, \cf  Eq. \eqref{eq:limitproblem} is unique.
\end{theorem}

\begin{proof}
In order to prove uniqueness, we assume that $(n_1, p_1)$ and $(n_2,p_2)$ are two solutions and let $\Omega$ be a compact, simply connected Lipschitz set that contains the union of their supports. Upon subtracting the equation for $n_2$ from the equation for $n_1$ we see that difference, $n_1 - n_2$, satisfies
\begin{equation}
    \label{eq: n1-n2}
    \partialt{\prt*{n_1-n_2}}-\Delta (p_1-p_2) - \nabla\cdot\prt*{(n_1-n_2) \nabla \Phi} - (n_1 G(p_1) - n_2 G(p_2))=0.
\end{equation}
For the sake of simplicity, we shall use the short-hand notation $G_i=G(p_i)$, for $i=1,2$, and $v=\nabla\Phi$. Multiplying Eq. \eqref{eq: n1-n2} by a test function $\psi=\psi(x,t)$ and integrating by parts we get
\begin{equation}
    \label{eq: integral}
    \iint_\OmegaT \left[(n_1-n_2)\partialt\psi +(p_1-p_2) \Delta \psi -(n_1-n_2) \nabla \psi \cdot\nabla v+(n_1 G_1-n_2 G_2) \psi \right] \dx{x}\dx{t}=0.
\end{equation}
The strategy is to employ Hilbert's dual method to establish uniqueness. To this end we introduce the following notation
\begin{align*}
    \left\{
    \begin{array}{rl}
     \curlyZ \rmspace &:= n_1 - n_2 + p_1 - p_2, \\[1em]
     \curlyA \rmspace &:= \dfrac{n_1-n_2}{ Z},\\[1em]
     \curlyB \rmspace &:= \dfrac{p_1 - p_2}{ Z},\\[1em]
     \curlyC \rmspace &:= - n_2 \dfrac{G_1 - G_2}{p_1 - p_2},
     \end{array}
     \right.
\end{align*}
where we set $ \curlyA =  \curlyB = 0$, whenever $\curlyZ =0$. Using this notation we rewrite Eq. \eqref{eq: integral} which becomes
\begin{equation}\label{eq: integral_phi}
      \iint_\OmegaT \curlyZ \left[  \curlyA \partialt\psi + \curlyB \Delta \psi - \curlyA \nabla \psi \cdot\nabla v + (\curlyA G_1- \curlyB \curlyC) \psi \right] \dx{x}\dx{t}=0. 
\end{equation}
Note that, by definition,
\begin{equation*}
  0 \leq   \curlyA, \curlyB \leq 1, \quad\text{as well as}\quad 0\leq  \curlyC \leq \sup_{0\leq p\leq p_M} |G'(p)|.
\end{equation*}
In order to apply Hilbert's duality method, we have to find a solution, $\psi$, to the \emph{dual problem}
\begin{equation}
    \label{eq: nonparab_SYS}
     \curlyA \partialt\psi + \curlyB \Delta \psi - \curlyA \nabla \psi \cdot\nabla v+(\curlyA G_1- \curlyB \curlyC) \psi = \curlyA \xi,
\end{equation}
in $\OmegaT$, and $\psi=0$ on $\partial\Omega \times (0,T)$. The equation is complemented by the final time condition 
$\psi(x,T)=0$ for $x\in \Omega$. Here, $\xi$ is an arbitrary smooth function. If solved, substituting the solution to the dual problem, $\psi$, into Eq. \eqref{eq: integral_phi} would yield
\begin{equation}
    \label{eq:formal_strategy}
    \iint_\OmegaT \curlyA \curlyZ \xi \dx{x} \dx{t} = \iint_\OmegaT (n_1-n_2) \xi \dx{x} \dx{t} = 0,
\end{equation}
thus proving uniqueness of the density. Subsequently, from Eq. \eqref{eq: integral}, the uniqueness of the pressure follows.

However, since the coefficient of Eq. \eqref{eq: nonparab_SYS} are not smooth and $A$ and $B$ can vanish, the equation is not uniformly parabolic and we need to regularise the system first. To this end, let $\{\curlyA_k\}, \{\curlyB_k\}, \{\curlyC_k\}, \{v_k\}$, $\{\curlyG_{1,k}\}$ be approximating sequences of smooth and bounded functions  such that
\begin{subequations}
\label{eq: reg_coeff}
\begin{align}
    \|\curlyA-\curlyA_k\|_{L^2(\Omega_T)}, \|\curlyB-\curlyB_k\|_{L^2(\Omega_T)}, \|\curlyC-\curlyC_k\|_{L^2(\Omega_T)}, \|G_1 - \curlyG_{1,k}\|_{L^2(\Omega_T)}, \|v-v_k\|_{L^2(\Omega_T)}\leq \frac1k,
\end{align}
such that
\begin{align}
     1/k \leq \curlyA_k, \curlyB_k \leq 1, \quad\text{as well as} \quad   0\leq \curlyC_k, |\curlyG_{1,k}|  \leq C,
\end{align}
and
\begin{align}
     &\|\partial_t \curlyC_k\|_{L^1(\Omega_T)},  \|\nabla \curlyG_{1,k}\|_{L^2(\Omega_T)}   \leq C,
\end{align}
where $C>0$ is some positive constant.
\end{subequations}
Using the regularised  quantities, we consider the regularised equation
\begin{equation}
    \label{eq: regular_SYS}
    \partialt{\psi_k} +\frac{\curlyB_k}{\curlyA_k} \Delta \psi_k - \nabla \psi_k \cdot\nabla v_k+ \prt*{ \curlyG_{1,k} - \frac{\curlyB_k \curlyC_k}{\curlyA_k}} \psi_k =  \xi,
\end{equation}
in $\OmegaT$, and $\psi_k=0$, on $\partial\Omega \times (0,T)$, and $\psi_k(T, x)=0$, in $\Omega$. Here, $\xi$ denotes an arbitrary smooth test function which is crucial for this approach, as discussed above, \cf Eq. \eqref{eq:formal_strategy}. Since the coefficient $\curlyB_k/\curlyA_k$ is smooth and bounded from away from zero, the equation is uniformly parabolic, whence we infer the existence of a smooth solution, $\psi_k$.

Using $\psi_k$ as a test function in Eq. \eqref{eq: integral_phi} and thanks to Eq. \eqref{eq: regular_SYS} we get
\begin{align*}
    0 &= \iint_\OmegaT \curlyZ \prt*{\curlyA\partialt{\psi_k} + \curlyB \Delta \psi_k - \curlyA \nabla v \cdot \nabla \psi_k + (\curlyA G_1 - \curlyB \curlyC)\psi_k}\dx{x}\dx{t}\\[1em]
    &= \iint_\OmegaT \curlyZ \curlyA \prt*{-\frac{\curlyB_k}{\curlyA_k}  \Delta \psi_k + v_k \cdot \nabla \psi_k - \prt*{\curlyG_{1,k}-\frac{\curlyB_k \curlyC_k}{\curlyA_k}}\psi_k +\xi }\dx{x}\dx{t}\\[1em]
    &\quad +\iint_\OmegaT \curlyZ \prt*{\curlyB \Delta \psi_k - \curlyA v \cdot \nabla \psi_k + (\curlyA G_1 - \curlyB \curlyC)\psi_k } \dx{x}\dx{t}\\[1em]
     &=\iint_\OmegaT \curlyZ \curlyA \xi + \iint_\OmegaT \curlyZ \frac{\curlyB_k}{\curlyA_k} \prt*{\curlyA - \curlyA_k}\prt*{-\Delta \psi_k + \curlyC_k \psi_k} \dx{x}\dx{t}\\[1em]
     &\quad + \iint_\OmegaT \curlyZ (\curlyB_k - \curlyB)(-\Delta \psi_k + \curlyC_k \psi_k) \dx{x}\dx{t} +\iint_\OmegaT \curlyZ \curlyB (\Delta \psi_k - \curlyC \psi_k)\dx{x}\dx{t} \\[1em]
     &\quad+ \iint_\OmegaT \curlyZ \curlyB (- \Delta \psi_k + \curlyC_k \psi_k)\dx{x} \dx{t} + \iint_\OmegaT \curlyZ \curlyA \psi_k (G_1 - \curlyG_{1,k}) \dx{x}\dx{t}\\[1em]
     &\quad+\iint_\OmegaT \curlyZ \curlyA \nabla \psi_k \cdot (v_k - v)\dx{x}\dx{t}.
\end{align*}
Using the definition of $\curlyA$, $\curlyB$, and $\curlyZ$, we finally obtain
\begin{equation*}
    \iint_\OmegaT (n_1-n_2) \xi \dx{x}\dx{t}  = I^1_k - I^2_k + I^3_k -  I^4_k + I^5_k,
\end{equation*}
where
\begin{align*}
    I^1_k =& \iint_\OmegaT (n_1-n_2 + p_1 -p_2) \frac{\curlyB_k}{\curlyA_k}(\curlyA - \curlyA_k)(\Delta \psi_k - \curlyC_k \psi_k)\dx{x}\dx{t},\\[1em]
    I^2_k =& \iint_\OmegaT (n_1-n_2+p_1-p_2) (\curlyB - \curlyB_k)(\Delta \psi_k - \curlyC_k \psi_k) \dx{x}\dx{t},\\[1em]
    I^3_k =& \iint_\OmegaT (p_1 - p_2) (\curlyC - \curlyC_k)\psi_k \dx{x}\dx{t},\\[1em]
    I^4_k =& \iint_\OmegaT (n_1-n_2)(G_1 - \curlyG_{1,k}) \psi_k \dx{x}\dx{t},\\[1em]
    I^5_k =& \iint_\OmegaT (n_1-n_2) \nabla\psi_k \cdot (v-v_k) \dx{x}\dx{t}.
\end{align*}
We aim at showing that 
$$
    \lim_{k\rightarrow\infty}I^i_k=0,
$$ 
for $i=1,\dots,5$, in order to be able to conclude that $n_1=n_2$. Before proving the convergence of each $I^i_k$, we need certain uniform bounds which we collect and state in the subsequent lemma.
\begin{lemma}[Uniform bounds]
There exist a positive constant $C>0$, independent of $k$, such that
\begin{align}
    \label{eq: bound on lap psin}
    \begin{split}
        \sup_{0\leq t\leq T} &\|\psi_k(t)\|_{L^\infty(\Omega)}\leq C, \quad \sup_{0\leq t\leq T}  \|\nabla\psi_k(t)\|_{L^2(\Omega)}\leq C,\\
        &\|(\curlyB_k/\curlyA_k)^{1/2}(\Delta \psi_k -\curlyC_k \psi_k)\|_{L^2(\OmegaT)}\leq C.
    \end{split}
\end{align}
\end{lemma}
\begin{proof}
The $L^\infty$-bound comes directly from the maximum principle applied to Eq. \eqref{eq: regular_SYS}, since $\xi$ is bounded and
$$
    \curlyG_{1,k} - \frac{\curlyB_k \curlyC_k}{\curlyA_k} \leq C.
$$
Now we multiply Eq. \eqref{eq: regular_SYS} by $(\Delta \psi_k - \curlyC_k \psi_k)$ and integrate in $(t,T)\times\Omega$ to obtain
\begin{align}
   \label{eq: final}
   \begin{split}
   - \int_t^T & \int_\Omega \partialt{}\frac{|\nabla \psi_k|^2}{2} \dx{x}\dx{s}- \int_t^T\int_\Omega \frac{\curlyC_k}{2}\partialt{}\psi_k^2 \dx{x}\dx{s} + \int_t^T\int_\Omega \frac{\curlyB_k}{\curlyA_k} |\Delta \psi_k - \curlyC_k \psi_k|^2\dx{x}\dx{s} \\[0.7em]
   &= \underbrace{\int_t^T\int_\Omega v \cdot \nabla\psi_k (\Delta \psi_k - \curlyC_k \psi_k)\dx{x}\dx{s}}_{\curlyI_1} + \underbrace{-\int_t^T\int_\Omega \curlyG_{1,k}\psi_k (\Delta \psi_k - \curlyC_k \psi_k )\dx{x}\dx{s}}_{\curlyI_2}\\[0.7em]
   &+ \underbrace{\int_t^T\int_\Omega \xi (\Delta \psi_k - \curlyC_k \psi_k )\dx{x}\dx{s}}_{\curlyI_3},
   \end{split}
\end{align}
where we shall bound each of the terms, $\curlyI_i$, for $i=1,2,3$, individually. First note that
\begin{align*}
    \curlyI_1 
    &= \int_t^T\int_\Omega v \cdot \nabla\psi_k \Delta \psi_k \dx{x}\dx{s} - \int_t^T\int_\Omega v \cdot \nabla\psi_k \curlyC_k \psi_k\dx{x}\dx{s}\\[0.7em]
    &= \curlyI_{1,1} + \curlyI_{1,2}.
\end{align*}
Integrating by parts in the first term of $\curlyI_1$ we get
\begin{align*}
    \curlyI_{1,1} &= - \int_t^T\int_\Omega \sum_{i,j=1}^d \frac{\partial v^{(i)}}{\partial x_j} \frac{\partial \psi_k}{\partial x_i} \frac{\partial\psi_n}{\partial x_j} \dx{x}\dx{s} - \int_t^T\int_\Omega \sum_{i,j=1}^d v^{(i)} \frac{\partial^2\psi_k}{\partial x_i \partial x_j} \frac{\partial \psi_k}{\partial x_j}\dx{x}\dx{s}\\[0.7em]
    &= -\int_t^T\int_\Omega \sum_{i,j=1}^d \frac{\partial v^{(i)}}{\partial x_j} \frac{\partial \psi_k}{\partial x_i} \frac{\partial\psi_n}{\partial x_j} \dx{x}\dx{s} + \int_t^T\int_\Omega  \frac{|\nabla \psi_k|^2}{2} \div v  \dx{x}\dx{s}\\[0.7em]
    &\leq  \prt*{d \|\nabla v\|_{L^\infty} + \frac 1 2 \|\nabla \cdot v\|_{L^\infty}} \int_t^T\int_\Omega |\nabla \psi_k|^2\dx{x}\dx{s},
\end{align*}
where $v^{(i)}$ is the $i$-th component of the vector $v$ and $\nabla v $ is the matrix with element $(\nabla v)_{i,j}=\partial_j v^{(i)}$. Similarly, we observe
\begin{align*}
    \curlyI_{1,2} &= - \int_t^T \int_\Omega v \cdot \nabla \psi_k \curlyC_k \psi_k \dx{x}\dx{s}\\[0.7em]
    &\leq \frac12 \|v\|_{L^\infty(\OmegaT)} \|\curlyC_k\|_{L^\infty(\OmegaT)} \|\psi_k\|_{L^2(\OmegaT)}^2 + \frac12\|\nabla \psi_k\|_{L^2(\OmegaT)}^2\\[0.7em]
    &\leq C + C\|\nabla \psi_k\|_{L^2(\OmegaT)}^2,
\end{align*}
with $C>0$ independent of $k$, after applying Young's inequality. Hence
\begin{align*}
    \curlyI_1 \leq  C + C\|\nabla \psi_k\|_{L^2(\OmegaT)}^2.
\end{align*}
Next, let us address the term $\curlyI_2$. We observe that
\begin{align*}
    \curlyI_2 &= -\int_t^T\int_\Omega \curlyG_{1,k}\psi_k (\Delta \psi_k - \curlyC_k \psi_k )\dx{x}\dx{s} \\[0.7em]
    &= \int_t^T\int_\Omega \curlyG_{1,k}|\nabla \psi_k|^2 \dx{x}\dx{s} +  \int_t^T\int_\Omega \psi_k \nabla\psi_k \cdot \nabla \curlyG_{1,k} \dx{x}\dx{s} + \int_t^T\int_\Omega \curlyG_{1,k} \curlyC_k \psi_k \dx{x}\dx{s}.
\end{align*}
We note that $\|\curlyG_{1,k}\|_{L^\infty(\Omega_T)}$ whence we obtain bounds for the first and the last term, respectively. In addition, we recall $\|\nabla \curlyG_{1,k}\|_{L^2(\Omega_T)}\leq C$, whence, upon using Young's inequality, we get
\begin{align*}
   \int_t^T\int_\Omega \psi_k \nabla\psi_k \cdot \nabla \curlyG_{1,k} \dx{x}\dx{s}&\leq \frac 1 2  \|\psi_k\|_{L^\infty(\Omega_T)}\|\nabla \curlyG_{1,k}\|^2_{L^2(\Omega_T)} +\frac 1 2 \|\psi_k\|_{L^\infty(\Omega_T)} \int_t^T\int_\Omega | \nabla\psi_k|^2 \dx{x}\dx{s}\\[0.7em]
   &\leq C + C \int_t^T\int_\Omega | \nabla\psi_k|^2 \dx{x}\dx{s}.
\end{align*}
In combination we get
\begin{align*}
    \curlyI_2 \leq C + C \|\nabla \psi_k\|_{L^2(\OmegaT)}^2,
\end{align*}
with $C>0$ independent of $k$. Last, let us address the term $\curlyI_3$. We readily observe
\begin{align*}
    \curlyI_3 &= \int_t^T\int_\Omega \xi (\Delta \psi_k - \curlyC_k \psi_k )\dx{x}\dx{s}\\[0.7em]
    &\leq C,
\end{align*}
integrating by parts twice and using the $L^\infty$-bounds. Using the bounds obtained above, the right-hand side of Eq. \eqref{eq: final} can be bounded as follows
\begin{align*}
    C + &C \|\nabla \psi_k\|_{L^2(\OmegaT)}^2 \\[0.7em]
    &\geq - \int_t^T \int_\Omega \partialt{}\frac{|\nabla \psi_k|^2}{2} \dx{x}\dx{s}- \int_t^T\int_\Omega \frac{\curlyC_k}{2}\partialt{}\psi_k^2 \dx{x}\dx{s} + \int_t^T\int_\Omega \frac{\curlyB_k}{\curlyA_k} |\Delta \psi_k - \curlyC_n \psi_k|^2\dx{x}\dx{s} \\[0.7em]
   &\geq - \int_t^T \frac{\dx{}}{\dx{t}}\int_\Omega \frac{|\nabla \psi_k|^2}{2} \dx{x}\dx{s} + \int_t^T\int_\Omega \partialt{\curlyC_k} \frac{\psi_k^2}{2} \dx{x}\dx{s} + \int_t^T\int_\Omega \frac{\curlyB_k}{\curlyA_k} |\Delta \psi_k - \curlyC_k \psi_k|^2\dx{x}\dx{s}\\[0.7em]
   &\qquad + \int_\Omega \frac{\curlyC_k(t)
   \psi_k^2(t)}{2}\dx{x}\\[0.7em]
   &\geq \frac12 \|\nabla \psi_k(\cdot, t)\|_{L^2(\Omega)}^2 - \|\partial_t \curlyC_k\|_{L^1(\OmegaT)} \|\psi_k\|_{L^\infty(\OmegaT)}^2 + \int_t^T\int_\Omega \frac{\curlyB_k}{\curlyA_k} |\Delta \psi_k - \curlyC_k \psi_k|^2\dx{x}\dx{s}\\[0.7em]
   &\qquad - \frac12\|\curlyC_k\|_{L^\infty(\OmegaT)} \|\psi_k\|_{L^2(\OmegaT)}^2\\[0.7em]
   &\geq \frac12 \|\nabla \psi_k(\cdot, t)\|_{L^2(\Omega)}^2  + \int_t^T\int_\Omega \frac{\curlyB_k}{\curlyA_k} |\Delta \psi_k - \curlyC_k \psi_k|^2\dx{x}\dx{s}- C,
\end{align*}
having used the regularity assumptions on the regularised coefficients, \cf Eq. \eqref{eq: reg_coeff}.

Finally, since $\curlyC_k$ is positive, we get
\begin{align}
    \label{eq: to_gronwall}
    \frac{1}{2} \int_\Omega |\nabla \psi_k(t)|^2\dx{x}  + \int_t^T\int_\Omega \frac{\curlyB_k}{\curlyA_k} |\Delta \psi_k - \curlyC_k \psi_k|^2\dx{x}\dx{s}\leq C + C \int_t^T\int_\Omega | \nabla\psi_k|^2 \dx{x}\dx{s}.
\end{align}
Introducing the notation
$$
    Q(s):=\int_\Omega |\nabla\psi_k(s,x)|^2 \dx{x},
$$
we observe that Eq. \eqref{eq: to_gronwall} now reads
\begin{equation*}
    Q(t) \leq C + C \int_t^T Q(s) \dx{s},
\end{equation*}
and by Gronwall's lemma we conclude that 
\begin{equation*}
    \sup_{0\leq t \leq T} Q(t)= \sup_{0\leq t \leq T}\|\nabla \psi_k(t)\|^2_{L^2(\Omega)}\leq C.
\end{equation*}
The third bound of Eq. \eqref{eq: bound on lap psin} comes a posteriori from Eq. \eqref{eq: to_gronwall}, which completes proof.
\end{proof}
Thanks to these uniform bounds, we obtain
\begin{align*}
    I^1_k 
    &= \iint_\OmegaT (n_1-n_2 + p_1 -p_2) \frac{\curlyB_k}{\curlyA_k}(\curlyA - \curlyA_k)(\Delta \psi_k - \curlyC_k \psi_k)\dx{x}\dx{t}\\[1em]
    &\leq C \|(\curlyB_k / \curlyA_k)^{1/2} (\curlyA - \curlyA_k)\|_{L^2(\Omega_T)}\\[1em]
    &\leq C k^{1/2} \|\curlyA - \curlyA_k\|_{L^2(\Omega_T)}\\[1em]
    &\leq C/k^{1/2},
\end{align*}
and, similarly,
\begin{align*}
    I^2_n &= \iint_\OmegaT (n_1-n_2+p_1-p_2) (\curlyB - \curlyB_k)(\Delta \psi_k - \curlyC_k \psi_k) \dx{x}\dx{t}\\[1em]
    &\leq C k^{1/2} \|\curlyB - \curlyB_k\|_{L^2(\Omega_T)}\\[1em]
    &\leq C / k^{1/2}.
\end{align*}
Finally, we have
\begin{align*}
    I^3_k &= \iint_\OmegaT (p_1-p_2) (\curlyC - \curlyC_n)\psi_k \dx{x}\dx{t}\\[1em]
    &\leq C \|\curlyC - \curlyC_k\|_{L^2(\Omega_T)}\\[1em]
    &\leq C / k,
\end{align*}
and
\begin{align*}
    I^4_k &= \iint_\OmegaT (n_1-n_2)(G_1 - \curlyG_{1,k}) \psi_n \dx{x}\dx{t}\\[1em]
    &\leq C\|G_1 - \curlyG_{1,k}\|_{L^2(\Omega_T)}\\[1em]
    &\leq C /k,
\end{align*}
as well as
\begin{align*}
    I^5_n &= \iint_\OmegaT (n_1-n_2) \nabla\psi_n \cdot (v-v_k) \dx{x}\dx{t}\\[1em]
    &\leq C\|v-v_k\|_{L^2(\Omega_T)}\\[1em]
    &\leq C / k.
\end{align*}
In summary, we have 
\begin{equation*}
    \iint_\OmegaT (n_1-n_2) \xi \dx{x}\dx{t}  = I^1_k - I^2_k + I^3_k -  I^4_k + I^5_k \longrightarrow 0,
\end{equation*}
as $k\to \infty$, and therefore $n_1=n_2$. From Eq. \eqref{eq: integral} we have
\begin{equation*}
    \iint_\OmegaT \prt*{(p_1-p_2)\Delta \psi + n_1 (G(p_1)-G(p_2))\psi}\dx{x}\dx{t} = 0.
\end{equation*}
Taking a smooth approximation of $p_1-p_2$ as test function we get
\begin{equation*}
    \iint_\OmegaT |\nabla (p_1-p_2)|^2 \dx{x}\dx{t} =\iint_\OmegaT  n_1 (G(p_1)-G(p_2))(p_1-p_2)\dx{x}\dx{t},
\end{equation*}
and, by the monotonicity of $G$, \cf Eq. \eqref{eq:assptn_growth}, we conclude that $p_1=p_2$.
\end{proof}

\section{Velocity of the boundary for patches}
\label{sec:velo_bdry}
Let us recall that the Hele-Shaw problem is given by
\begin{equation}
\label{eq:HSgeometric}
\left\{
\begin{array}{rll}
    -\Delta p_\infty \!\!\! &= \Delta \Phi +G(p_\infty), &\text{ in } \Omega(t),\\[0.7em]
    V\!\!\! &= -\prt*{\nabla p_\infty + \nabla \Phi} \cdot \nu,  &\text{ on } \partial \Omega(t),
\end{array}
\right.
\end{equation}
where $\nu$ indicates the outward normal to the boundary and $\Omega(t):=\{x; \ p_\infty(x,t)>0\}$. Below we give a characterisation of patch solutions, \ie, the indicator of the growing domain described by Eq. \eqref{eq:HSgeometric} satisfies the incompressible limit equation, \cf Eq. \eqref{eq:limitproblem}. To this end, we suppose that the boundary $\partial\Omega(t)$ admits a Lipschitz parameterisation $\partial\Omega(t)=\{x(t,\alpha) \, | \, \alpha\in [0,1], x(t,0)=x(t,1)\}$ that satisfies
\begin{equation}\label{eq: para_velocity}
    \ddt{x(t, \alpha)}= -(\nabla p_\infty(x(t, \alpha),t)+\nabla \Phi(x(t,\alpha),t)).
\end{equation}

Then the characteristic function
\begin{equation}\label{eq: characteristic function}
    n_\infty(t)=\mathds{1}_{\Omega(t)}.
\end{equation}
satisfies the limit problem, Eq.  \eqref{eq:limitproblem}.

\begin{theorem}[Characterisation of the Free Boundary Velocity]
Let $\Omega_0$ be a bounded and Lipschitz continuous domain. Let us consider the solution $(\Omega(t), p_\infty)$ to the free boundary problem, Eq. \eqref{eq:HSgeometric}, with initial data $\Omega_0$. Then, the characteristic function in Eq. \eqref{eq: characteristic function}, satisfies Eq. \eqref{eq:limitproblem}.
\end{theorem}
\begin{proof}
We have to show that $n_\infty(t)= \mathds{1}_{\Omega(t)}$ satisfies
\begin{equation*}
    \partialt{n_\infty}=\Delta p_\infty + \nabla\cdot(n_\infty\nabla\Phi)+n_\infty G(p_\infty),
\end{equation*}
in the distributional sense.
Given a test function $\psi=\psi(x)$, by Reynolds' transport Theorem and Eq. \eqref{eq: para_velocity}, we have
\begin{equation*}
   \int_{\R^d} \psi(x) \partialt {n_\infty} \dx{x} = \ddt \int_{\R^d} \psi(x) \mathds{1}_{\Omega(t)}\dx{x} = \int_{\partial\Omega(t)}  V \psi(x) \dx{x} = V \delta_{\partial\Omega(t)}.
\end{equation*}
On the other hand, it holds
\begin{equation*}
    \Delta p_\infty + \nabla\cdot (n_\infty \nabla\Phi) +n_\infty G(p_\infty)=-(\partial_\nu p_\infty +\partial_\nu \Phi )\delta_{\partial \Omega(t)}= V \delta_{\partial\Omega(t)},
\end{equation*}
in the sense of distributions, as can be seen by the following argument. First, by the definition of $\Omega(t)$ as the positivity set of $p_\infty$ and the fact that $n_\infty = \mathds{1}_{\Omega(t)}$ we observe that the weak formulation of the left-hand side can be manipulated as follows:
\begin{align*}
    &\int_{\R^d} -\nabla p_\infty \cdot\nabla \psi - n_\infty\nabla \Phi \cdot \nabla \psi + n_\infty G(p_\infty)\psi\dx{x} = \int_{\Omega(t)}-\nabla p_\infty \cdot\nabla \psi - \nabla \Phi \cdot \nabla \psi + G(p_\infty)\psi\dx{x}.
\end{align*}
Integrating by parts the right-hand side, we obtain
\begin{align*}
    \int_{\Omega(t)}& (\Delta p_\infty + \Delta \Phi+ G(p_\infty))\psi \dx{x} -\int_{\partial \Omega(t)}\partial_\nu p_\infty\psi\dx{x}-\int_{\partial \Omega(t)}\partial_\nu \Phi\psi\dx{x}\\
    &\quad = -\int_{\partial \Omega(t)}\partial_\nu p_\infty\psi \dx{x}-\int_{\partial \Omega(t)}\partial_\nu \Phi\psi\dx{x}
\end{align*}
where we used $\Delta p_\infty + \Delta \Phi+ G(p_\infty)=0,$ in $\curlyD'$, by Eq. \eqref{eq:HSgeometric}.
\end{proof}


\section*{Acknowledgements}
This project has received funding from the European Union's Horizon 2020 research and innovation program under the Marie Skłodowska-Curie (grant agreement No 754362). The authors are grateful for the discussions with Beno\^it Perthame at several occasions.

\bibliography{literature}
\bibliographystyle{plain}
\end{document}